\documentclass[11pt,reqno]{amsart}
\pdfoutput=1
\usepackage{amsmath}%\
\usepackage{amsfonts}%
\usepackage{amssymb}%
\usepackage{amsthm}
\usepackage{xspace}

\usepackage{float}

\usepackage{multirow}% http://ctan.org/pkg/multirow
\usepackage{graphicx, import}
\usepackage{tikz-cd}
\usepackage{hyperref}
\usepackage{cleveref}
\usepackage{a4wide}

\usepackage{subcaption}

\usepackage{parskip}

\crefrangelabelformat{equation}%Writing multiple references as (1a-1c)
{(#3#1#4--#5\crefstripprefix{#1}{#2}#6)}
\crefrangelabelformat{subequation}%
{(#3#1#4--#5\crefstripprefix{#1}{#2}#6)}

\newtheorem{theorem}{Theorem}[section]
\newtheorem*{mainthm*}{\hypertarget{thm:DRclosure}{\bf Main theorem}}
\newtheorem{theorem+definition}{Theorem\,+\,Definition}[theorem]

\numberwithin{case}{theorem}

\newtheorem*{claim*}{Claim}

\newtheorem{lemma}[theorem]{Lemma}

\newtheorem{proposition}[theorem]{Proposition}

\newtheorem{def+prop}[theorem]{Definition\,+\,Proposition}
\theoremstyle{definition}
\newtheorem{definition}[theorem]{Definition}
\theoremstyle{remark}
\newtheorem{remark}[theorem]{Remark}
\newtheorem{example}[theorem]{Example}
\DeclareMathOperator{\CC}{\mathbb{C}}

\DeclareMathOperator{\PP}{\mathbb{P}}
\DeclareMathOperator{\ZZ}{\mathbb{Z}}

\DeclareMathOperator{\mult}{mult}

\DeclareMathOperator{\im}{Im}

\DeclareMathOperator{\ord}{ord}

\DeclareMathOperator{\ANN}{Ann}

\DeclareMathOperator{\ev}{ev}

\newcommand\calO{\mathcal{O}}

\newcommand\calC{\mathcal{C}}

\newcommand\calM{\mathcal{M}}

\newcommand\calH{\mathcal{H}}

\newcommand\Adm[1][(\mu)]{\operatorname{Adm}_{g,n}#1 }
\newcommand\Hur[1][(\mu)]{\operatorname{Hur}_{g,n}#1 }

\newcommand{\Mgn}[1][g,n]{\calM_{#1}}

\newcommand{\Mgnbar}[1][g,n]{ \overline\calM_{#1}}

\newcommand{\Cgn}[1][g,n]{\calC_{#1}}

\newcommand{\HB}{E}
\newcommand{\PHB}[1][]{\PP_{#1}(\HB)}
\newcommand{\Stra}[1][(\mu)]{\calH {#1}}
\newcommand{\DR}[1][(\ptn)]{\operatorname{DR}_g \!{#1}}
\newcommand{\DRg}[2][g]{\operatorname{DR}_{#1}{#2}}
\newcommand\dr{double ramification locus\xspace}
\newcommand\drs{double ramification loci\xspace}

\newcommand\twr{twistable rational function\xspace}
\newcommand\twrs{twistable rational functions\xspace}
\newcommand\twdr{twisted rational function\xspace}
\newcommand\twdrs{twisted rational functions\xspace}

\newcommand\TWR{(\stC,\umS,\stf)}
\newcommand\TWDR{(\psC,\umSC,\psf)}

\newcommand{\DRD}[1][(\ptn)]{\widetilde{\operatorname{DR}}_g{#1}}
\newcommand{\DRDc}[1][(\eptn)]{\widetilde{\operatorname{DR}}_g^c{#1}}
\newcommand{\DRc}[1][(\eptn)]{\operatorname{DR}_g^c{#1}}

\newcommand{\PStra}[1][(\mu)]{\PP\!\Stra[#1]}

\newcommand{\MSDS}[1][(\mu)]{\Xi\Mgnbar {#1}}

\newcommand{\LVLF}[1][\leq i]{L_{#1}}

\newcommand{\psC}{\widetilde{X}}
\newcommand{\stC}{X}
\newcommand{\stf}{f}
\newcommand{\psf}{\tilde{f}}

\newcommand{\exD}{d\stf}

\newcommand{\twD}{\eta}
\newcommand{\rleq}[2][i]{{#2}_{(\leq #1)}}

\newcommand{\req}[2][i]{{#2}_{(#1)}}

\newcommand{\pa}{\gamma}

\newcommand\mS[1][k]{x_{#1}}%marked section
\newcommand\mZ[1][k]{z_{#1}}%\marked zeros
\newcommand\mC[1][k]{c_{#1}}%marked critical points
\newcommand\mSC[1][k]{\tilde{x}_{#1}}
\newcommand\mN[1][e]{q_{#1}}

\newcommand\umS{\mathbf{x}}%marked section
\newcommand\umP{\mathbf{p}}% marked poles
\newcommand\umZ{\mathbf{z}}%\marked zeros
\newcommand\umC{\mathbf{c}}%marked critical points
\newcommand\umSC{\mathbf{\tilde{x}}}
\newcommand\umN{\mathbf{\mN[]}}

\newcommand{\dG}[1][\Gamma]{#1}% standard symbol for dual graph
\newcommand{\lG}[1][\Gamma]{\overline{\dG[#1]}}% standard notation for

\newcommand\ptn{\mu}%partition for DR loci
\newcommand\dptn{\mu'}% partition for df without critical points
\newcommand\eptn{\widetilde{\mu}}% partition including critical points

\tikzset{
  symbol/.style={
    draw=none,
    every to/.append style={
      edge node={node [sloped, allow upside down, auto=false]{$#1$}}}
  }
}

\newlength{\mylen}
\setbox1=\hbox{$\bullet$}\setbox2=\hbox{\tiny$\bullet$}
\numberwithin{equation}{section}

\title{The closure of double ramification loci via strata of exact differentials}
\author{Frederik Benirschke}
\address{Mathematics Department, Stony Brook University,
Stony Brook, NY 11794-3651, USA}
\email{frederik.benirschke@stonybrook.edu}

\begin{document}
\maketitle
\begin{abstract}
Double ramification loci, also known as strata of $0$-differentials, are algebraic subvarieties of the moduli space of smooth curves parametrizing Riemann surfaces such that there exists a rational function with prescribed ramification over $0$ and $\infty$. We describe the closure of double ramification loci inside the Deligne-Mumford compactification in geometric terms.

To a rational function we associate its exact differential, which allows us to realize double ramification loci as linear subvarieties of strata of meromorphic differentials. We then obtain a geometric description of the closure using our recent results on the boundary of linear subvarieties. Our approach yields a new way of relating the geometry of loci of rational functions and Teichm\"uller dynamics. We also compare our results to a different approach using admissible covers.
% Using our recent results on the boundary of linear subvarieties, we obtain an obstruction to smoothing out a collection of rational functions given by the vanishing of the evaluation morphism. The evaluation morphism sums up the values of the rational functions at certain marked points and nodes of the stable curve. We prove that the vanishing of the evaluation morphism is both necessary and sufficient to smooth out a collection of rational functions, while staying inside the double ramification locus.
%
%\Sam{I think overall you would want to make the abstract half the lengths. I would almost eliminate the definition, much shorten the discussion of the evaluation morphism. The last sentence is definitely too technical for an abstract. Would you be willing to write a new completely rewritten version?}
\end{abstract}

\section{Introduction}
Inside the moduli space $\Mgn$ of pointed smooth curves $(\stC,\umS)$ of genus $g$ with $n$ marked points $ \umS=(\mS[1],\ldots,\mS[n])$ for any integer $k\geq 0$ there exist natural closed subvarieties
\[
\calH^{k}_g(\ptn):=\left\{(X,\umS)\,:\, \calO_X\left(\sum_{i=1}^{n}\mu_i\mS[i]\right)\simeq\omega^{\otimes k}_X\right\}\subseteq \Mgn,
\]
where $\ptn=(\mu_1,\ldots,\mu_n)\in\ZZ^n$ is an integer partition of $k(2g-2)$. 
%The codimension of $\calH^{k}_g(\ptn)$ is $g$.

For $k\geq 1$, the above condition is equivalent to the existence of a meromorphic $k$-differential on $\stC$ with prescribed vanishing order at the marked points. In this case $\calH^{k}(\ptn)$ is called a {\em stratum of meromorphic $k$-differentials of type $\mu$} and they have been studied extensively, especially from the viewpoint of Teichm{\"u}ller dynamics. See for example the surveys \cite{ChenSurvey, WrightSurvey, ZorichSurvey}.
Strata of $k$-differentials are not compact, and finding geometrically meaningful compactifications is an active area of research.
One possible way of compactifying the stratum $\calH^{k}_g(\ptn)\subseteq \Mgn$ is by simply taking the closure inside $\Mgnbar$.  The closure of $\calH^{k}_g(\ptn)$ in $\Mgnbar$ has been determined for $k=1$ in \cite{BCGGMgrc} in terms of {\em twisted differentials}, and for $k>1$ in \cite{BCGGMk}, by using a covering construction to reduce it to the $k=1$ case.
In \cite{BCGGMsm} the authors use the explicit description of the closure of $\calH^{k}_g(\ptn)$ to construct the {\em moduli space of multi-scale differentials} $\MSDS$: a smooth, modular compactification of the stratum $\calH^1_g(\ptn)$, which has been extended to $k> 1$ in \cite{ecStrata}.

In the case $k=0$, the isomorphism $ \calO_X\left(\sum_{i=1}^{n}\mu_i\mS[i]\right)\simeq\calO_X$ is equivalent to the existence of a rational function $f:X\to \PP^1$ with prescribed ramification  over $0$ and $\infty$ with the preimages of $0$ and $\infty$  being all marked points $x_k$ such that $\mu_k>0$ and $\mu_k<0$, respectively. The subvariety $\DR:=\calH_g^{0}(\ptn)$ of $\calM_{g,n}$ is called a {\em \dr} or a {\em stratum of $0$-differentials}. Equivalently, \drs can be defined as the pullback of the zero-section from the universal Jacobian under the Abel-Jacobi map. In 2001 Eliashberg posed the problem of extending double ramification cycles to $\Mgnbar$ for the development of symplectic field theory \cite{SymplField}.
Since then various different ways of extending the cycle class of $\DR$ to $\Mgnbar$ have been studied in the literature. One possible extension is via relative Gromov-Witten theory using the space of rubber maps to $\PP^1$, see \cite{Li1,Li2,GV}.
A different approach to extending \drs is to extend the Abel-Jacobi map.
The Abel-Jacobi naturally extends to curves of compact types. The corresponding cycle class was computed by Hain in \cite{Hain} and the computation was later simplified in \cite{GZ}. However, in general the Abel-Jacobi map does not extend to  all of $\Mgnbar$; nonetheless different extensions of \drs to $\Mgnbar$ using the Abel-Jacobi map have been proposed, see for example \cite{HKPDR,HolmesDR,DRLog}. Another extension has been proposed in \cite{FarkasPandharipande}.
There is one natural cycle class on $\Mgnbar$ that arises in these approaches, and a conjectural formula for the class in the tautological ring of $\Mgnbar$ was proposed by Pixton and later proved in \cite{DRFormula}. We stress that the proposed extensions do not coincide with the cycle class of the closure of $\DR$ in $\Mgnbar$ but rather contain additional contributions from the boundary.

%\Sam{erase "Instead"; but really erase the entire sentence, as it's repeated one paragraph down} Instead in this paper we take a different approach and describe the closure of $\DR$ inside $\Mgnbar$ explicitly in geometric terms.

One way of finding  a geometric description of the closure of $\DR$ is using the theory of admissible covers, first introduced by \cite{HM}. The space of admissible covers $\Adm$ is a proper Deligne-Mumford stack compactifying the space of maps $X\to\PP^1$ with a fixed ramification profile $\sigma$. We will recall admissible covers in more detail in \Cref{sec:Adm}. Contributions to the intersection  theory of admissible cover have been made in \cite{SchmittvanZelm, LianHtautological}.
Using $\Adm$ a simple description of the boundary of double ramification loci is as follows: a stable curve $(X,\umS)\in \Mgnbar$ is contained in $\overline{\DR}$ if and only if there exists a suitable admissible cover on a semistable model $X'$ of $X$. While certainly geometric, the existence of an admissible cover on a semistable model is hard to verify in practice due to the combinatorial complexity. In particular it would be useful to have a description solely in terms of the given stable curve $X$.

To obtain a more explicit geometric criterion in this paper we take a different approach, similar in spirit to the approach for $k\geq 1$ in \cite{BCGGMgrc, BCGGMk} by relating double ramification loci to strata of meromorphic differentials. By associating to a rational function $f:\stC\to\PP^1$ the exact differential $df$ we can use the moduli space of multi-scale differentials $\MSDS$, constructed in \cite{BCGGMsm}, to understand degenerations of $df$. This gives a new way of relating  double ramification loci to strata of differentials. Since exact differentials are described by the vanishing of all absolute periods, this requires an analysis of periods near the boundary of $\MSDS$. The results obtained in \cite{Fred} and further refined in \cite{BDG} describe the boundary   of subvarieties of strata of meromorphic differentials given by linear equations on periods in $\MSDS$, and can thus be applied to the case of exact differentials.

We now define the key notions, which will allow us to state the main result.
\subsection*{Twistable rational functions}
Let $\dG$ be the dual graph of a marked stable curve.
We consider $\dG$ as a triple $(V,E,H)$ where $V$ are the  vertices, $E$ are the edges and $H$  are the {\em legs} corresponding to the marked points. We denote $h_k$ the leg corresponding to the marked point $\mS$.

Recall from \cite{BCGGMgrc} that a {\em level graph} $\lG$ is a dual graph $\dG$ of a stable curve together with a level function $\ell:V(\dG)\to \ZZ_{\leq 0}$. A level function induces a partial order on the set of vertices by setting  $v \succcurlyeq v'$ if $\ell(v)\geq\ell(v')$. Here we understand partial ordering in a weak sense and allow $\ell(v)=\ell(v')$ even if $v\neq v'$.

\begin{definition}\label{def:twr}
Let $\ptn$ be a partition of zero as above, and $\lG$ a level graph.
A {\em \twr} $(\stC,\umS,\stf)$ of type $\ptn$ compatible with $\lG$ on a stable pointed curve $(\stC,\umS)$ is a collection $\stf=(\stf_v)_{v\in V(\dG)}$ such that $f_v:\widetilde{X}_v\to\PP^1$ is a rational function on the normalization $\tilde{X}_v$ of each irreducible component $X_v$ of $X$, satisfying
\begin{enumerate}
\item ({\bf Prescribed order of vanishing}) Each rational function $f_v$ is holomorphic  away from the marked points and nodes. Furthermore,
\[
\ord_{\mS}\stf = \ptn_k \text{ for all } k.
\]
If a component $X_v$ contains any marked {\em zero} $x_k$ of $f_v$, then all zeros of $f_v$ are at marked points or nodes.

\item ({\bf Matching order at nodes}) If a node of $\stC$ identifies $\mN[1]\in \tilde{\stC}_{v_1}$ with $\mN[2]\in\tilde{\stC}_{v_2}$.
Then
\begin{equation}\label{eq:TWR}
\begin{aligned}
% &\mult_{q_1} \stf_{v_1} + \mult_{q_2} \stf{v_2} \geq 0 && \text{ if } \ord_{q_1} \stf\geq 0,\, \ord_{q_2} \stf\geq 0,\\
% &\mult_{q_1}\stf -\mult_{q_2} \stf \geq 0 && \text{ if } \ord_{q_2} \stf< 0.
 \ord_{q_1} (d\stf_{v_1}) + \ord_{q_2} (d\stf_{v_2})\geq -2
\end{aligned}.
\end{equation}

\item
({\bf Compatibility with the level graph}) If furthermore $q_2$ is a pole of $f_{v_2}$, i.e. $\ord_{q_2} \stf< 0$, then $\ell(v_1)> \ell(v_2)$.
\end{enumerate}
\end{definition}

Some remarks about this definition are in order.
We note that condition $(2)$  is vacuous if $\ord_{q_1} \stf\geq 0$ and $\ord_{q_2} \stf\geq 0$. If $\ord_{q_2} \stf <0$ then it can be rephrased as
\[
\mult_{q_1} f_{v_1} \geq \mult_{q_2} f_{v_2},
\]
where $\mult_q f_v-1$ is the ramification index of the map $f_v:X_v\to \PP^1$ at $q$.

\begin{remark} \label{rem:unmarked}
We stress that while all poles of $\stf$ are either at a marked point $\umS$ or at a node, there can be zeros of $\stf$ which are not marked and not at a node.
There are thus two kinds of irreducible components depending on whether or not they contain marked zeros of $\stf_v$.   On an irreducible component containing marked zeros, the pointed curve $(\tilde{X}_v,(\umS_v,\umN_v))$, where $\umS_v$ and $\umN_v$ are all the marked points and nodes contained in $\tilde{X}_v$, is contained in a double ramification locus $\DRg[g(X_v)](\ptn_v)$ for a partition $\ptn_v$ depending only on $\ptn$ and the level graph $\lG$ and the orders of vanishing of $f$ at the nodes. On the remaining components, which do not contain any marked zero of $f_v$, we instead have a rational function with conditions on the ramification indices at the nodes.
\end{remark}

\subsection{The evaluation morphism of a \twr}\label{sec:ev}
The definition of \twrs so far is similar to the definition of a twisted differential in the sense of \cite{BCGGMgrc}, except that in $(2)$ we only have inequality. We now introduce the crucial new ingredient that will allow us to determine whether a \twr can be smoothed to a rational function contained in a \dr.

Let $(\stC,\umS,\stf)$ be a \twr of type $\ptn$ compatible with a level graph $
\lG$. We partition $\umS=(\umZ,\umP)$ where all marked points in $\umZ$ and $\umP$  are zeros and poles  of $\stf$, respectively.
From now on we usually assume that  $\mu_i\neq 0$ at all marked points. All our results can be extended to the case where $\mu_i=0$ at some marked point.  In this case one need to include all marked points with $\mu_i=0$ in $\umZ$.
%We usual and all marked points in $\umP$ are poles of $\stf$.
We consider the dual graph $\lG$ with legs corresponding to the marked points $\umS$ as a $1$-dimensional cell complex. By abuse of notation we call $\mS$ the endpoint of the leg $h(\mS)$ and we can thus consider $\umS$ as a subset of $\lG$.

For a given level $i\in\ZZ_{\le 0}$ we denote $\rleq{\lG}$ the subgraph consisting of all vertices $v$ such that $\ell(v)\leq i$, together with all edges between them.

We define $\req{\lG}$ similarly, except that for every edge $e$ connecting a vertex $v$ of level $i$ and a vertex of level lower than $i$,  we attach an additional leg $h(\mN^+)$ to $v$. Pictorially, we cut each edge connecting a vertex of level $i$ to a vertex of lower levels in the middle.

We say a relative homology class $[\gamma]\in H_1(\lG,\umZ;\ZZ)$  has {\em top level} at most $i$ if it can be represented by a collection of paths $\gamma$ contained in $\rleq{\lG}$. We then  let $\LVLF(\lG)\subseteq H_1(\lG,\umZ;\ZZ)$  be the subspace of all homology classes of level at most $i$. We can now describe the {\em evaluation morphism}
\[
\ev_{\stf}^{(i)}:\LVLF(\lG)\to \CC
\]
of level $i$ as follows.
Suppose that $[\gamma]\in \LVLF(\lG)$ is represented by a collection of paths $\gamma$ contained in $\rleq{\lG}$.
We then restrict $\gamma$ to $\req{\lG}$ and evaluate $f$, with signs, at the endpoints of the restriction.
For example, if the restriction of $\gamma$ connects $q_1^+$ and $q_2^+$, as depicted in \Cref{fig:evmap}, then we have
\[
\ev_{\stf}^{(0)}([\gamma])=f(q_2^+)-f(q_1^+)\,,
\]
and if the restriction of $\gamma$ is a closed path then we have $\ev_{\stf}^{(i)}([\gamma])=0$.
\begin{figure}
\includegraphics[scale=1.3]{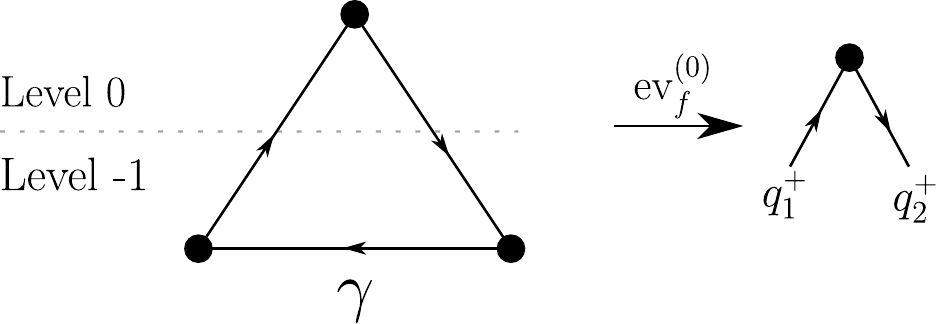}
\caption{The evaluation map}
\label{fig:evmap}
\end{figure}
We will revisit the evaluation map in more detail in \Cref{sec:evmap} and in particular check that it is well-defined in \Cref{prop:evwell}. We can now state our main result.

\begin{mainthm*} Let $(X,\umS)$ be a stable curve in $\Mgnbar$ with dual graph $\dG$. Then $(X,\umS)$ is contained in the closure of $\DR$ if and only if the following conditions are satisfied:
\begin{enumerate}
\item There exists a level graph structure $\lG$ on $\dG$ and a \twr $(X,\umS,\stf)$ of type $\ptn$ compatible with $\lG$.
\item  The evaluation morphism $\ev_{\stf}^{(i)}$ vanishes identically  for all levels $i$.
\end{enumerate}
\end{mainthm*}

Being a \twr with vanishing evaluation morphism is a condition that can be verified explicitly. On each irreducible component one has a rational function satisfying a condition similar to double ramification loci, and the vanishing of the evaluation morphism is a combinatorial condition relating the values of the rational functions on different irreducible components.

In \Cref{sec:ex} we discuss several examples of \twrs highlighting the features of the evaluation morphism, showing the explicit nature of \twrs. The Abel-Jacobi map extends to curves of compact type and thus degenerations of the \dr in this case are well understood. Furthermore, in \cite{GZ} the authors extend the Abel-Jacobi map to curves with one non-separating node and use it do describe the closure of \dr.
Our techniques allow us to describe arbitrary degenerations. As a demonstration, going beyond what has been studied in the literature,  in \Cref{ex:Dollar} we study \twrs on dollar curves, i.e. stable curves with two irreducible components meeting transversely at three nodes. We discuss the possible level graphs and the different types of conditions that the vanishing of the evaluation morphism imposes.

At first glance the notion of a twistable rational function is very similar to the notion of twisted differentials in \cite{FarkasPandharipande} and \cite{BCGGMgrc}.
A twisted differential is a collection of  meromorphic differential forms on each irreducible component. In particular, it determines a unique partial order on the dual graph by comparing the vanishing orders at the nodes.
In \cite{BCGGMgrc} the authors introduce the global residue condition for twisted differentials, which is the crucial ingredient that allows to smooth out a twisted differential. The global residue condition does depend not only the partial ordering but on the actual level graph, as can be seen in \cite[Ex 3.4]{BCGGMgrc}.
The situation is somewhat different for \twrs.
In \Cref{ex:paOrder} we will see that a \twr does not determine a unique partial ordering, i.e. there exist two different partial orders on the dual graph of the same underlying \twr both satisfying Condition $(3)$ in \Cref{def:twr}.
%The definition of a \twr only uses a partial ordering of the vertices of the dual graph. This should be compared to the notion of twisted differentials as in \cite{FarkasPandharipande}.
%A twisted differential induces a unique partial ordering on the dual graph.
The level graph only comes into play when considering the evaluation morphism. There are several cases where a fixed partial ordering on the dual graph of a \twr uniquely determines the level graph (up to automorphism) and thus also determines the conditions imposed by the vanishing of the evaluation morphism, for example if every irreducible component contains a marked zero of $\stf$.
On the other hand in \Cref{ex:lvldep} we construct an example of two different level graphs, inducing the same partial ordering, for which the vanishing of the  evaluation morphism imposes different conditions.
%We were not able to find an example where the vanishing of the evaluation morphism actually depends on the level graph and not just the partial ordering.
%\Sam{Here I would say more precisely that the limiting function is actually different, but the conditions it imposes on the components do not change in the examples you have considered}
%\Fred{at the moment I can't even come up with an example like this}

\begin{remark}({\em Prestable and stable curves})
For technical reasons, we will have to additionally mark the critical points of the rational functions during the proof of the Main theorem. By forgetting the marked critical points we are then naturally forced to consider prestable models of a given stable curves. A combinatorial observation (\Cref{lemma:LocMax}) allows us to state our results only in terms of the stable curve without alluding to a prestable model. On the other hand for admissable covers it seems hard in general to give a criterion for when a given stable curve admits an admissable cover on a semistable model.
\end{remark}

\subsection*{Potential applications}
The determination of the closure of strata inside the Hodge bundle via twisted differentials in \cite{BCGGMgrc} allowed to study the geometry of strata of differentials with algebro-geometric methods. For example, in \cite{QuentinThesis} the author obtains first information about the Kodaira dimension of strata using twisted differentials. In \cite{MUWReal} the authors solve the realization problem for tropical canonical divisors using the smoothability of twisted differentials.
Using the smoothability of \twrs, in forthcoming work we plan to study similar questions for the \dr.
Additionally, the description of the closure of strata was used in \cite{BCGGMsm} to construct a smooth, modular compactification of strata, and one would like to apply similar methods to construct a  smooth, modular compactification of \dr, different from the normalization of the stack of admissable covers, as described in \cite{AVC}.

\subsection*{Outline of the proof}
The key idea to prove the \hyperlink{thm:DRclosure}{Main theorem}  is to consider exact differentials instead of rational functions. On the \dr $\DR$ only the zeros and poles of $\stf$ are marked, but not the remaining zeros of $d\stf$, which correspond to the critical points of $\stf$. In \Cref{sec:exdiff},  we construct analogs of $\DR$, which we denote $\DRc$, by additionally marking the critical points.
The advantage of using $\DRc$ is that the associated exact differential is contained in a fixed stratum. Furthermore, being an exact differential is equivalent to the vanishing of all absolute periods, a condition that is linear in period coordinates of the stratum.
We can then describe the boundary of $\DRc$ in terms of {\em \twdrs} using the main result from \cite{Fred} and some further consequences from \cite{BDG}.
A \twdr is a \twr where we additionally mark all the critical points of the rational function, and such that the associated exact differential is a twisted differential in the sense of \cite{BCGGMgrc}.

In \Cref{prop:BDDRc} we then show that the boundary of $\DRc$  can be described in terms of \twdrs with vanishing evaluation morphism.

So far we have artificially marked the critical points of $\stf$ and now we need to forget them again.
It then only remains to study the relationship between \twdrs and \twrs.
In \Cref{sec:twisting} we show that \twrs arise exactly  by forgetting the critical points of a \twdr and then stabilizing the resulting prestable curve, which is then enough to prove the main theorem.

\subsection*{Acknowledgments}
I would like to thank my advisor Samuel Grushevsky for numerous useful discussions regarding this project. Furthermore, I would like to thank Johannes Schmitt for comments on a preliminary draft.

\section{Exact differentials and \twdrs}\label{sec:exdiff}

\subsection*{Level graphs}
We recall some additional notation for level graphs from \cite{BCGGMgrc}.
 An edge of $\lG$ is called {\em horizontal} if it joins vertices of the same level, and {\em vertical} otherwise.
If an edge $e$ joins the vertices $v$ and $v'$ such that $\ell(v)\geq \ell (v')$ then we let $\ell(e+)$ be the level of $v$ and similarly $\ell(e-)$ the level of $v'$. Furthermore we set $v(e+):=v$ and $v(e-):=v'$. Similarly, we denote $q_e^+$ and $q_e^-$ the preimages of the node $e$ on $v(e+)$ and $v(e-)$ respectively. At horizontal nodes we make a random choice. Similarly, for a leg $h$ we let $v(h)$ be the vertex connected to $h$ and $\ell(h)$ be the level of $v(h)$.

\subsection*{Moduli spaces of exact differentials}

Let $X$ be a stable curve with dual graph $\dG$ and $f=(f_v)_{v\in V(\dG)}$ a collection of rational functions on the irreducible components of $X$. To $f$ we can associate the collection of meromorphic differentials $df=(d(f_v))_{v\in V(\dG)}$.
We have
\begin{equation}\label{eq:orddf}
\ord_{x} (\exD)= \begin{cases} \mult_x f-1 & \text{ if } x\notin f^{-1}(\infty) \\
-\mult_x f-1 & \text{ if } x\in f^{-1}(\infty).
\end{cases}
\end{equation}
where $\mult_x \stf$ is the ramification index of the map $\stf_v:X_v\to\PP^1$ at $x$.
To study degenerations of $f$ we want to use the compactification of strata of meromorphic differentials given by multi-scale differentials. In order to ensure that $df$ is contained in a fixed stratum we need to control the zero and pole order of $\stf$ as well as the ramification.

We define the projectivized Hodge bundle
\[\PHB:=\PP\left(\pi_*\omega_{\Cgn/\Mgn}\left(\sum_{\mu_k <0}\left(-\mu_k\mS\right )\right )\right),\]
 where $\pi:\Cgn\to\Mgn$ denotes the universal curve.
Since exact differentials are exactly the meromorphic differentials with zero absolute periods, we can identify $\DR$ with a locally closed subset of $\PHB$ as follows.
We define a new partition $\dptn $ by setting  $\dptn_k:= \ptn_k-1$
and call $\dptn$ the partition {\em associated} to $\ptn$. Then the locally closed subset
\[
\DRD:=\left\{(X,\umS,\omega)\,:\,  \int_{\pa} \omega=0\,\forall \pa\in H_1(X\setminus \umP,\umZ;\ZZ),\, \ord_{\mS}\omega=\dptn_k \text{ for all } k\right\}\subseteq \PHB
\]
is isomorphic to $\DR$.

\subsection*{Marking the critical points}
We will need analogs of $\DRD$ where all critical points are marked and the zero order of $df$ is prescribed at all zeros, nodes and critical points of $f$.
Let $\TWR$ be a \twr. We denote $\umC$ the remaining critical points of $\stf$ that are neither zeros, poles or nodes, and set
\[
\umSC=(\umZ,\umP,\umC).
\]
We stress that a marked point is always a smooth point of $\psC$, thus $\umS$ consists of zeros and poles of $f$ that are not nodes, and $\umC$ consists of  all critical points of $\psf$ that are not zeros or poles of $f$ or nodes.

We fix a partition $\ptn$ of zero of length $n$. We say a partition $\eptn$ of $2g-2$ {\em extends} $\ptn$ if it can be written as  $\eptn=(\dptn,\dptn')$ such that $\dptn$ is the partition associated to $\ptn$ and all entries of $\dptn'$ are positive. We denote the length of $\eptn$ by $n'$.
Let
\[
\DRDc:=\left\{(X,\umSC,\omega)\,:\,  \int_{\pa} \omega=0\,\forall \pa\in H_1(X\setminus \umP,\umZ;\ZZ),\, \ord_{\mSC}\omega=\eptn_k \text{ for all } k\right\}\subseteq \PStra[(\eptn)]
\]

We similarly let $\DRc$ be the image of $\DRDc$ under the forgetful map $\PStra[(\eptn)]\to \Mgn[g,n']$.

We now define a notion analogous to \twr, taking into account also the total ramification of $\stf$.
\begin{definition}
Let $\eptn$ be a partition extending $\mu$. A  pointed stable curve $\TWDR$ with a collection $d\psf=(d\psf_v)_{v\in V}$ of rational functions is called a {\em \twdr} if $d\psf$ is a twisted differential of type $\eptn$ compatible with $\lG$ in the sense of \cite{BCGGMgrc}.
\end{definition}

Let us unravel this definition. First, note that $d\psf$ cannot have simple poles and thus $\lG$ cannot have any horizontal edges. Therefore the matching residue condition is vacuous. Furthermore, since $d\psf$ is exact, all of its periods over absolute homology classes are zero and thus it has no residues at higher order poles either. In particular the global residue condition is also vacuous. Thus equivalently we can define a \twdr as follows.

\begin{definition}
Let $\eptn$ be as above.
We call a pointed stable curve $\TWDR$ together with a collection of rational functions and a decomposition $\umSC=(\umS,\umC)$ a {\em twisted rational function} of type $\eptn$  compatible with $\lG$ if the following conditions are satisfied.
\begin{enumerate}
\item ({\bf Prescribed order of vanishing}) Each rational function $\psf_v$ is holomorphic  away from the marked points and nodes. If some marked section $\mS$ is a zero for $\psf_v$, then $\psf_v$ is non-zero away from the marked points and nodes. Furthermore, $\ord_{\mSC} (d\psf) =\eptn_k$  for all $k$.
\item ({\bf Matching order at nodes}) Suppose a node of $\stC$ identifies $\mN[1]\in \tilde{\stC}_{v_1}$ with $\mN[2]\in\tilde{\stC}_{v_2}$, then
\[\ord_{\mN[1]} (d\psf_{v_1}) + \ord_{\mN[2]} (d\psf_{v_2})=-2\]

\item ({\bf Compatibility with the level graph}) Furthermore,  $\ell(v_1)> \ell(v_2)$ if and only if $\ord_{q_2} (d\psf_{v_2})<0$.

\end{enumerate}
\end{definition}
We call $\umS$ the {\em marked zeros and poles} of $\psf$ and $\umC$ the {\em remaining critical points}.
% where $\umS$ are zeros and poles of $\psf$ and $\umC$ are the remaining critical points of $\psf$ that are not zeros or poles of $\psf$ or nodes.
 Sometimes we want to decompose further and write $\umS=(\umZ,\umP)$ where $\umZ$ and $\umP$ are zeros and poles of $\psf$, respectively.
We stress that, as in the definition of \twr, not all zeros of $\psf$ have to be marked points in $\umS$ or nodes, since there can be irreducible components not containing any marked point of $\umS$.

The reason for using \twdrs instead of \twrs is that we can use the smoothing results from \cite{Fred,BDG}.

\subsection{The evaluation morphism}\label{sec:evmap}
Let  $(\psC,\umSC,\psf)$ be a \twdr with marked points $\umSC=(\umZ, \umP, \umC)$. The embedding $(\rleq{\lG},\umZ)\to(\lG,\umZ)$ induces a natural map
 \[
 H_1(\rleq{\lG},\umZ;\ZZ)\to H_1(\lG,\umZ;\ZZ).
 \]
  We define the {\em level filtration} $\LVLF[\bullet](\lG)$ of $H_1(\lG,\umZ;\ZZ)$ by
\begin{equation}\label{eq:LVLF}
\LVLF(\lG):=\im\left(H_1(\rleq{\lG},\umZ;\ZZ)\to H_1(\lG,\umZ;\ZZ)\right)
\end{equation}
and we say that a cycle $[\gamma]\in \LVLF(\lG))\setminus \LVLF[\leq i-1](\lG))$ has {\em top level} $i$.

We are now going to define the evaluation morphism
\[
\ev_{\psf}^{(i)}:\LVLF(\lG))\to \CC.
\]
 We recall that $\req{\lG}$  has additional legs $h(q_e^+)$ corresponding to nodes with $\ell(e+)=i,\ell(e-)<i$. By embedding the leg $h(q_e^+)$ into the edge $e$ we can consider $\req{\lG}$ as a subspace of $\rleq{\lG}$.  We stress that this is a continuous map of cell complexes and not a graph morphism.

The idea to construct $\ev_{\psf}^{(i)}$ is to take a path in $\rleq{\lG}$ and restrict it to $\req{\lG}$. Let $[\gamma]\in \LVLF(\lG)$ be represented by a collection of smooth paths $\gamma$ which are contained in $\rleq{\lG}$. The restriction of $\gamma$ to $\req{\lG}$ decomposes into a sum of disjoint paths $\alpha_1,\ldots, \alpha_m$ which are either closed or have endpoints at the marked zeros or preimages of nodes $\umZ\cup\umN_{(i)}^+$. Here $\umN_{(i)}^+$ are the half-legs corresponding to nodes with $\ell(e+)=i,\ell(e-)<i$. We then define
\[
\ev_{\psf}^{(i)}:=\sum_{k=1}^m \psf(\alpha_k(1))-\psf(\alpha_k(0)).
\]
If all $\alpha_k$ are closed paths this is automatically zero.
Note that it follows from condition $(3)$ in the definition of \twdr that $\psf(q_e^+)$ is always finite.

 For a \twr $\TWR$ we define $\ev_{\stf}^{(i)}$ in the same way, where we decompose $\umS=(\umZ,\umP)$.

\begin{proposition}\label{prop:evwell}
 Suppose $\psf$ is a \twr or a \twdr, then the map $\ev_{\psf}^{(i)}:\LVLF(\lG)\to \CC$ is well-defined.
\end{proposition}
 \begin{proof}
 The proof is similar to \cite[Prop. 4.2]{Fred} for $2$-dimensional cell complexes instead of graphs.
From the long exact sequence of the triple
\[
\umZ\subseteq \rleq[i-1]{\lG}\cup \umZ\subseteq \rleq[i]{\lG}\cup \umZ
\]
we obtain the exact sequence
\[
 \begin{tikzcd}
 H_1(\rleq[i-1]{\lG},\umZ;\ZZ)\ar[r] &  H_1(\rleq[i]{\lG},\umZ;\ZZ) \ar[r,"\nu_i"] &  H_1(\rleq[i]{\lG}, \rleq[i-1]{\lG}\cup
\umZ;\ZZ)\ar[d,"\simeq"] \\
 & & H_1(\req{\lG},\umN_{(i)}^+\cup\umZ;\ZZ)
 \end{tikzcd}
 \]
 where the vertical isomorphism is induced by excising $\rleq[i-1]{\lG}$. We denote by
 \[
 g_i: H_1(\rleq[i]{\lG},\umZ;\ZZ)\to H_1(\req{\lG},\umN_{(i)}^+\cup\umZ;\ZZ)\]
  the composition of $\nu_i$ with the excision map. For a path $\gamma$ in $\rleq{\lG}$ the map $g_i$ is given by $g_i(\gamma)= \sum_{k} \alpha_k$ where $\sum_{k}\alpha_k$ is the restriction of $\gamma$ to $\req{\lG}$, considered as paths with endpoints in $\umZ\cup\umN^+_{(i)}$.
 Furthermore, there exists a boundary map
 \[
 \delta_i:H_1(\req{\lG},\umZ \cup \umN_{(i)}^+;\ZZ) \to H_0(\umZ\cup\umN_{(i)}^+;\ZZ)
 \]
 from the long exact sequence of the pair $(\req{\lG},\umZ \cup \umN_{(i)}^+)$ which sends a path $\alpha$ to $[\alpha(1)]-[\alpha(0)]$.
 We define
 \[
 \begin{split}
\ev: H_0(\umZ\cup\umN^+;\ZZ)&\to\CC,\\
\sum_{k}c_k[\mZ]+\sum_{e}d_e[\mN^+]&\mapsto \sum_{k}c_k\psf(\mZ)+\sum_{e}d_e\psf(\mN^+).
\end{split}
 \]
 Then we have
 \[
\ev_{\psf}^{(i)} =\ev\circ\ \delta_i\circ g_i
 \]
 and thus $\ev_{\psf}^{(i)}$ is well-defined.
 \end{proof}

We are now able to characterize the closure of $\DRc\subseteq \Mgnbar[g,n']$ using the results of \cite{Fred}.

\begin{proposition}\label{prop:BDDRc}
A marked stable curve $(\psC,\umSC)$ is contained in the closure of $\DRc$ in $\Mgnbar[g,n']$ if and only if there exists a level graph $\lG$ for $\psC$ and a collection of rational functions $\psf$ on $\psC$ such that $\TWDR$ is a
\twdr compatible with $\lG$ and $\ev_{\psf}{(i)}=0$ for all $i$.
\end{proposition}

\begin{proof}
Let $(\psC,\umSC)$ be  in the closure of $\DRc$. Then there exists a level graph $\lG$ on $\psC$ and a twisted differential $\twD$ on $\psC$ compatible with $\lG$ such that $(\psC,\umSC,\twD)$ is in the closure of $\DRDc$ inside $\Xi\Mgn[g,n'](\eptn)$. Note that $\twD$ cannot have residues, since every residue is the limit of an absolute period $\int_{\gamma} {\omega}$ for $\gamma\in H_1(X\setminus\umP;\ZZ)$ along a family of differentials $(X,\omega)$ degenerating to $(\psC,\twD)$. In particular, $\lG$ has  no horizontal nodes.

 The subvariety $\DRDc\subseteq\PStra[(\eptn)]$ is a linear subvariety cut out by the equations
  \[
  \int_{\pa} \omega=0\quad\forall \pa\in H_1(\stC_0\setminus \umP_0,\umZ_0;\ZZ),
  \]
  near a point $(\stC_0,{\umSC}_0,\omega)\in \DRDc$. We have \[
T_{(\stC_0,\tilde{\mathbf{x}}_0,\omega)} \DRDc= \ANN(H_1(\stC_0\setminus \umP_0,\umZ_0;\ZZ)).
\]
%Thus by \cite[Thm.~1]{Fred} the boundary is again linear and the equations can be described as follows.
Thus by \cite[Prop. 10.1]{Fred} the intersection of $\DRDc$ with the open boundary stratum corresponding to $\lG$ is a linear subvariety given by $\ANN\left(\oplus_{i=-L(\lG)}^{0} \im \tilde{g}_i\right)$. Here the map
$\tilde{g}_i:H_1(\rleq{X}\setminus \umP,\umZ;\ZZ)\to H_1(\req{X}\setminus \umP,\umZ\cup\umN^+_{(i)};\ZZ)$ can be described as follows.
For a simple closed curve $\gamma \in \rleq{X}$, we have $\tilde{g}_i(\gamma)=\sum_{k}\alpha_k$ where $\sum_{k} {\alpha_k}$ is the restriction of $\gamma$ to $\req{X}$, considered as paths with endpoints in $\umZ\cup\umN^+_{(i)}$.
In particular we have $\int_{\gamma} \twD=0$ for any cycle contained in $X_v\setminus \umN$ and thus $\twD$ is a collection of exact differentials. Thus there exists a collection of rational functions $\psf$ on $\psC$ such that $d\psf=\twD$.
Furthermore for any $\gamma\in H_1(X\setminus \umP,\umZ)$ we have
\[
0=\int_{\tilde{g}_i(\gamma)} \twD =\int_{\tilde{g}_i(\gamma)} d\psf=\sum_{k=1}^m\int_{\alpha_k}d\psf=\sum_{k=1}^m \psf(\alpha_k(1))-\psf(\alpha_k(0))= \widetilde{\ev}_{\psf}^{(i)}(\gamma).
\]
It remains to show that if an irreducible component $X_v$ contains a marked zero $\mZ$, then $\psf_v$ is non-zero away from the marked points $\umS$ and nodes. We fix a marked zero $z_0$ of $\psf_v$ and let $z'$ be a zero  on the same component that is not a node and  is not marked.
Note that $z'$ is necessarily a simple zero and we can thus construct a holomorphic section marking $z'(t)$ along a family $(X(t),\omega(t))$ degenerating to $(\psC,\twD)$. Suppose that $(X(t),\omega(t))$ is a holomorphic family of exact differentials degenerating to $\twD$.
Then \[
\int_{z_0(t)}^{z'(t)}\omega(t)=t\cdot h(t)\]
where $h(t)$ is analytic in a neighborhood of the origin and thus the corresponding rational function $f(t)$ on $X(t)$ has a zero in a small neighborhood of $z'(t)$, disjoint from all marked zeros of $f(t)$. But since on $X(t)$ {\em  all } zeros of $\psf$ are marked, this is a contradiction.
This shows that $\TWDR$ is a \twdr compatible with $\lG'$.

It remains to show that we can smooth out a \twdr, while staying inside a \dr.
Let $\TWDR$ be a \twdr with vanishing evaluation morphism and $\twD=d\psf$ the associated twisted differential.
We denote $D\subseteq \MSDS$ the boundary of the moduli space of multi-scale differentials.
Let $U$ be a small open neighborhood of $(\psC,\umSC,\twD)$ in $\MSDS$. On $U\setminus D$ we can define a {\em local} analytic subvariety by  \[
 Z:=\left\{(\stC,\omega)\,:\, \int_{\pa} \omega=0\quad\forall \pa\in H_1(\stC\setminus \umP,\umZ;\ZZ)\right \}\subseteq U\setminus D.
  \]
By \cite[Cor 4.3]{BDG}, $Z$ is smooth and transverse to the boundary and we can thus smooth out $\TWDR$. Note that even though $Z$ is only defined locally, the proof of (loc. cit.)
applies verbatim in this situation.
Alternatively, the smoothness of $Z$ also follows from the proof of \cite[Prop 8.9]{Fred}.
\end{proof}

 \section{Twisting a \twr}\label{sec:twisting}
So far we have obtained a smoothing result for \twdr and now we need to transfer it to a smoothing result for \twrs. In this section we thus collect various results relating \twrs and \twdrs.

\begin{definition}Let $\TWDR$ be a \twdr of type $\eptn$ compatible with $\lG'$, where $\eptn$ extends $\ptn$. We recall that $\umSC=(\umS,\umC)$ consists of some marked zeros and poles of $\psf$ as well as the remaining critical points, that are neither zeros, poles of $\psf$ or nodes. In particular we can consider $\umS$ as a subset of $\umSC$. We consider the prestable curve $(\psC,\umS)$ where we forget the remaining critical points and let $(X,\umS)$ be its stabilization. Here we abused notation and  denote the images of $\umS$ under the stabilization map $\psC\to\stC$ again by $\umS$. The stabilization induces a graph morphism $\phi:\dG'\to\dG$ contracting edges that correspond to unstable chains of rational curves. The dual graph $\dG$ inherits a level graph structure $\lG$ from $\lG'$ using the stabilization morphism $\phi$. More concretely, for a vertex of $\lG'$ corresponding to a stable component we define $\ell(\phi(v)):=\ell'(v)$, where $\ell$ and $\ell'$ are the level functions of $\lG$ and $\lG'$, respectively. Furthermore, $\stC$ inherits a collection of rational functions $(f_v)_{v\in V(\dG)}$.
We call $\TWR$ {\em the stabilization } of $\TWDR$ and similarly say that $\lG$ is the stabilization of $\lG$.
We call any component of $\psC$ that is contracted under stabilization {\em unstable}.
Similarly starting with a \twr $\TWR$, if there exists a \twdr $\TWDR$ such that $\TWR$ is the stabilization of $\TWDR$, then we call $\TWDR$ a {\em twist} of $\TWR$.
\end{definition}

From now on $\lG'$ will always denote the level graph of a \twdr and $\lG$ the level graph of its stabilization.

The goal of this section is to show that we have the following correspondence: The stabilization of a \twdr is a \twr and every \twr admits a twist. We are now going to make this precise.

 \begin{lemma}\label{lm:ContrTW}Let $\eptn$ be a partition extending $\ptn$. Furthermore, let $\TWDR$ be a \twdr of type $\eptn$ compatible with a level graph $\lG'$. Then its stabilization $\TWR$ is a \twr of type $\ptn$ compatible with $\lG$.
 \end{lemma}
 \begin{proof}

Since $\TWDR$ is a \twdr, we have
\[ \ord_{\mS}\psf = \eptn_k \text{ for all }k,\, \ord_{\mN^+} d\psf+ \ord_{\mN^-} d\psf= -2 \text{ for all nodes } e\in E.
\]
The stabilization $\TWR$ is obtained by contracting all unstable components.
Chains of unstable components come in two flavors. A chain either connects two nodes, in which case we call it a {\em rational bridge}, or it forms a {\em rational tail}.

We first study the contraction of a rational tail. If the end of the rational tail does not contain a marked zero or a marked pole of $\psf$, then we do not mark the preimage of the node where the rational tail connects after stabilization and thus there is nothing to check at this point. We now claim that the end of a rational tail never contains a marked zero or a marked pole. This is part of the following \Cref{lemma:LocMax}. Thus, if an irreducible component $X_v$ contains a marked point $\mS$, then $X_v$ is not an unstable component and $\ord_{\mS} \stf= \ord_{\mS} \psf=\ptn_k$.

It remains to analyze the case of a rational bridge. Suppose we have a chain $C_1,\ldots, C_n$ of rational curves connecting two nodes preimages of nodes $q$ and $q'$.
Let $\mC[1],\ldots,\mC[m]$ be the marked critical points contained in $C_1,\ldots, C_n$. Then we have
\[
\ord_{q} d\psf + \ord_{q'} d\psf= -2+\sum_{k=1}^m \ord_{\mC} d\psf> -2,
\]
which proves the matching order at nodes for a \twr.
We also need to show the compatibility with the level graph, i.e. if $\ord_{q'} \stf<0$ then $\ell(v(q)) > \ell(v(q'))$.
By the following \Cref{lemma:LocMax},  no component $C_k$ of the rational bridge is a local maximum for the level order and thus the level can only increase along the chain $C_1,\ldots, C_n$ of rational curves.  Since $\lG'$ has  no horizontal nodes, the level is strictly increasing and thus $\ell(v(q))>\ell(v(q')$.
 \end{proof}

\begin{lemma}\label{lemma:LocMax} Let $\TWDR$ be a \twdr.
An unstable component $X_v$ does not contain a marked zero or pole of $\psf_v$.
Additionally, an unstable component with two nodes is not a local maximum for the level order.
\end{lemma}
\begin{proof}
We first assume there exists an unstable component containing a marked zero or pole $x$ of $\psf_v$.  Then there exists an unstable rational tail $C_1,\ldots,C_n$ with $x\in C_n$ and a stable component $C_0$ which is separated from $C_1$ by a node $q$. We let $q_0$ be the preimage of $q$ contained in $C_1$.
Note that $C_n$ has a single node. If $x$ is a zero of $\psf_{C_n}$, then $\psf_{C_n}$ has a pole at  the node.  Furthermore, since all zeros of $\psf_{C_n}$ are marked, $\psf_{C_n}$ has exactly one zero and one pole and thus no other critical points. This yields a contradiction since every unstable component contains a marked critical point.

We can thus assume that $x$ is a pole for $\psf_{C_n}$.  Let then $c_1,\ldots,c_m$ be the collection of  marked critical point on  the rational tail $C_1,\ldots,C_n$. Then
\[
\ord_{q_0} d\psf= \ord_{x} d\psf+\sum_{k=1}^m \ord_{\mC} d\psf> \ord_{x} d\psf.
\]
Here we used that the order of $d\psf$ is positive at all critical points in $\umC$. Note that in particular $\ord_{q} \psf> \ord_x \psf$. We now remove the rational tail $C_1,\dots,C_n$ and if $q_0$ is a pole for $\psf$ we mark it on the resulting curve $X'$. Furthermore, $X'$ inherits a level graph by contracting all the edges corresponding to nodes in the rational tail $C_1,\ldots,C_n$. We claim that the evaluation morphism still vanishes for $(X',\umS,d\psf)$.
 The dual graph $\lG$ of $(X$, $\umZ)$ deformation retracts onto the dual graph $\lG'$ of $(X',\umZ)$, where the deformation retract can be chosen to fix the zeros $\umZ$.
 Let $\alpha:(\lG,\umZ)\to(\lG',\umZ)$ be the resulting retract. Then $\alpha(\LVLF(\lG))=\LVLF(\lG')$ and  the diagram
 \[
\begin{tikzcd}[row sep=large, column sep=large]
 \LVLF(\lG) \ar[dr,"\ev_{f}^{(i)}(\lG)"]\ar[d,swap,"\alpha"] &\\
  \LVLF(\lG') \ar[r,swap,"\ev_{f}^{(i)}(\lG')"] & \CC
 \end{tikzcd}
 \]
 commutes, see also the proof of \Cref{lm:evCD} for a similar argument. This finishes the proof of the claim.

Thus, by what was already proven in \Cref{prop:BDDRc},  $(X',\umS,d\psf)$ can be smoothed out to an exact differential where all marked points $\umZ$ stay zeros of $d\psf$ of order $\ord_{\mZ}d\psf$. But this exact differential comes from a rational function with more zeros than poles, since $\ord_{q_0} \psf> \ord_{x}\psf$ as noted above, which is impossible.

It remains to prove the second part. Assume that some unstable component $X_v$ contains two nodes and is a local maximum for the level order. Then $X_v$ has no poles at the nodes and thus needs to have a marked pole somewhere. Thus $X_v$ contains two nodes and at least one marked pole and is not unstable, which is a contradiction.
\end{proof}

\begin{remark}
The above lemma is crucial for us. It is the main reason why we are able to state the \hyperlink{thm:DRclosure}{Main theorem} only in terms  of a given \twr instead of having to say that there exists a \twdr that stabilizes to the given \twr.
\end{remark}

So far we established that stabilizing a \twdr yields a \twr. We now show that conversely every \twr admits a twist.

 \begin{proposition}\label{prop:twistEx}
Let $\TWR$ be a \twr of type $\mu$ compatible with a level graph $\lG$. Then there exists a partition $\eptn$ extending $\mu$, a level graph $\lG'$ and a \twdr $\TWDR$ of type $\eptn$ compatible with $\lG'$ such that $\TWDR$ is a twist of $\TWR$.
 \end{proposition}
 \begin{proof}
As a first step we mark all zeros of $d\stf$, that are not already marked points in $\umS$, and declare them as critical points in $\umC$. To turn $\TWR$ into a \twdr we are going to insert rational curves at all nodes with $\ord_{\mN^+} d\stf + \ord_{\mN^-} d\stf>-2$.
%At a marked point $\mS$ with $\ord_{\mS} d\stf> \eptn_k$, we attach a marked rational curve $(\PP^1,0,1,\infty)$ at $\infty$. We choose a rational function $f'$ on $\PP^1$ such that the differential $df'$ vanishes with order $\eptn_k$ at $0$, vanishes with order $-\ord_{\mS} d\stf-2$ at $\infty$ has no other poles. Note that such a rational function always exists, since $\ord_{\mS} d\stf> \eptn_k$. We then declare all remaining zeros besides $0$ and $\infty$ of $df'$ to be marked critical points in $\umC$.
At such a node we insert a rational curve $(\PP^1,(0,1,\infty))$ by identifying $\mN^+$ and $\mN^-$ with $0$ and $\infty$, respectively. We  equip $\PP^1$ with a rational function $f'$  such that $df'$ has  order $-\ord_{\mN^+} df-2$ at $0$ and order $-\ord_{\mN^-} df-2$ at $\infty$ and no other poles. If both  $\ord_{\mN^+} df$ and $\ord_{\mN^-} df$ are positive, $f'$ has poles at $0$ and $\infty$ and thus such $f'$ exists. If $\ord_{\mN^-} df<0$, then $\ord_{\mN^+} d\stf>0$ by the compatibility with the level graph of a \twr. In particular $f'$ has a pole at $0$ and a zero at $\infty$.
 Note that in this case such a rational function exists since
 \[
\mult_{0} f' =\ord_{q^+} df+1>  -\ord_{q^-} df-1=\mult_{\infty}f'.
 \]
 We declare all remaining critical points of $f'$ to be marked points in $\umC$. Note that in both cases no zero of $f'$ is marked as a point in $\umS$.

 We let $(\psC,\umSC)$ be the resulting curve obtained by inserting the rational curves, where we set $\umSC:=(\umS,\umC)$. Furthermore by construction there exists a collection of rational functions $\psf$ on $\psC$.
Our next goal is to describe the level graph on $\psC$.
We start with the level function for $\lG$ and add additional levels in between two levels $i$ and $i-1$ and call those {\em intermediate levels}. For each level there are two intermediate levels, the level $i^{+}$ is in between $i+1$ and $i$ and the level $i^-$  is in between $i$ and $i-1$. Note that we have $(i-1)^+=i^-$ by construction. We then assign the level of an inserted rational curve as follows.

If $\mN^+$ is a pole of $\stf$, then the level of the inserted $\PP^1$ needs to be larger than $\ell(\mN^+)$ and otherwise the level of the inserted $\PP^1$ needs to be lower than $\ell(\mN^-)$. The same is true for $\mN^-$. Note that $\mN^+$ and $\mN^-$ cannot both be poles of $\stf$ because of the compatibility with the level graph in \Cref{def:twr}.
Thus if $\mN^+$ is a pole, we assign the level to be $\ell(\mN^+)^+$ and if $\mN^-$ is a pole we assign it to be $\ell(\mN^-)^+$. If both $\mN^+$ and $\mN^-$ are zeros of $\stf$ we declare the level to be $\min(\ell(\mN^+),\ell(\mN^-))^-$.
We denote $\lG'$ the resulting level graph on $\psC$ and by construction $\TWDR$ is compatible with $\lG'$.

To summarize, we have constructed a \twdr $\TWDR$  where $\umSC=(\umS,\umC)$ and $\umS$ consists of  zeros and  poles of $\psf$ and $\umC$ contains all remaining critical points of $\psf$ which are not nodes of $\psC$, together with a level graph $\lG'$ for $\psC$. We let $\eptn$ be the partition  consisting of the order of vanishing of $df$ at all marked points $\umSC$. Then $\eptn$ extends $\ptn$ and furthermore $\TWDR$ is of type $\eptn$ and compatible with $\lG'$. After forgetting all additional critical points $\umC$ and stabilizing we recover $\TWR$.
 \end{proof}

Let $\TWDR$ be a twist of $\TWR$.
As a final ingredient we need to compare the evaluation morphisms of $\stf$ and $\psf$.
The stabilization morphism $(\psC,\umSC)\to(\stC,\umS)$ induces a surjective map
$p:H_1(\lG',\umZ)\to H_1(\lG,\umZ)$.
Recall here that by abuse of notation we call the legs of $\lG'$ and of $\lG$ corresponding to the marked zeros of $\psf$ and $\stf$ both $\umZ$.

\begin{lemma}\label{lm:evCD}
The inclusion $p(\LVLF(\lG'))\subseteq \LVLF(\lG)$ holds and furthermore the diagram
\[
\begin{tikzcd}
\LVLF(\lG')\ar[d,"p"] \ar[dr,"{\ev}^{(i)}_{\psf}"] & \\
\LVLF(\lG) \ar[r,swap,"\ev_{\stf}^{(i)}"] & \CC
\end{tikzcd}
\]
commutes.
In particular, $\ev_{\stf}^{(i)}$ vanishes identically  if and only if ${\ev}_{\psf}^{(i)}$ vanishes identically.
\end{lemma}

\begin{proof}
The stabilization morphism removes some irreducible components but does not change the level of any stable component. Thus the top level can only go down under $p$, i.e. $ p(\LVLF(\lG'))\subseteq \LVLF(\lG)$.

Let $\gamma\in \LVLF(\lG')$. If the evaluation morphism $\ev_{\psf}^{(i)}(\gamma)$ has a contribution from a rational function $f_v$, then necessarily $\ell(v)=i$. Our goal is to show that the restriction of $\gamma$ to any unstable component $X_v$ of level $i$ is null-homotopic.
 Since the level graph $\lG'$ is purely vertical, $v$ has to be a local maximum for the level order on $\lG'$. Furthermore, $X_v$ cannot have two nodes by \Cref{lemma:LocMax}. This forces $X_v$ to be the end a rational tail, which does not have a marked zero, again by \Cref{lemma:LocMax}.
 But then the restriction of $\gamma$ to $X_v$ is null-homotopic. Thus $\gamma$ is homotopic to a path that does not cross any unstable top level component and thus $\ev_{f}^{(i)}([\gamma])=\ev_{\psf}^{(i)}([\gamma])$.

\end{proof}

\section{The proof of the Main theorem}

\begin{proof}[Proof of the {\hyperlink{thm:DRclosure}{Main theorem}}] We first prove necessity of the conditions; that is we need to show that given a marked curve $(X,\umS)$ in the closure of $\DR$ there exists a
\twr compatible with some level graph structure on the dual graph of $X$ with vanishing evaluation map.

Choose a holomorphic map $\phi:\Delta\to \Mgnbar$  from the unit disk $\Delta$ with $\phi(\Delta^*)\subseteq \DR$ and $\phi(0)=(\stC,\umS)$. We denote $\alpha$ the identification of $\DR$ and $\DRD\subseteq\PHB$. After shrinking $\Delta$, the image of $\alpha\circ\phi$ inside $\PHB$ is  contained in a stratum $\PStra[(\eptn)]$ for some partition $\eptn$ extending $\mu$. Here $\PStra[(\eptn)]$ denotes an unmarked stratum, i.e. not all zeros of the differential form are marked points.
We are now going to lift $\alpha\circ\phi$ to the moduli space of multi-scale differentials $\PP \Xi\Mgn[g,n'](\eptn)$.

{\em Claim:} After a base change $\Delta^*\to\Delta^*, z\mapsto z^k$ for a suitable $k$, we can construct a lift $\phi':\Delta\to \PP \Xi\Mgn[g,n'](\eptn)$ of $\alpha\circ\phi$ such that $\phi'(\Delta^*)\subseteq \DRc$.

{\em Proof of the claim.}
Since the number of critical points is constant, we can locally mark the critical points over $\Delta^*$.
Due to monodromy when encircling the origin in $\Delta$, those local sections might not be well-defined and might permute the critical points non-trivially. We then choose $k$ divisible enough so that the $k$-th power of the permutation is trivial. If we denote $\psi:\Delta^*\to \Delta^*, z \mapsto z^k$, then the family of curves corresponding to $\alpha\circ\phi\circ\psi$ admits global sections marking the critical points and thus there exists a map
$\phi':\Delta\to \PP \Xi\Mgn[g,n'](\eptn)$ such that $\phi'(\Delta^*)\subseteq \DRc$ and
 \[
\begin{tikzcd}
& \PP \Xi\Mgn[g,n'](\eptn)\ar[d,"\pi"] \\
\Delta\ar[ur,"\phi'"]\ar[r,swap,"\alpha\circ\phi\circ\psi"] & \Mgnbar
\end{tikzcd}
 \]
 commutes. Now the claim is proven.

We denote $(\psC,\umSC,\twD):=\phi  '(0) \in \PP \Xi\Mgn[g,n'](\eptn)$. Since $\pi(\psC,\umSC,\twD)=\phi(0)= (\stC,\umS)$, we conclude that  $\stC$ is the stabilization of $\psC$.  By \Cref{prop:BDDRc}, we can write $\twD=d\psf$ such that $\TWDR$ is a \twdr of type $\eptn$ compatible with a level graph $\lG'$ and $\ev_{\psf}^{(i)}=0$ for all $i$.
By \Cref{lm:ContrTW} the stabilization $\TWR$ of $\TWDR$ is a \twr of type $\mu$ compatible with $\lG$, where $\lG$ is the level graph for $\stC$ obtained by stabilizing $\lG'$. We furthermore conclude that $\ev^{(i)}_{\stf}=0$ for all $i$ by \Cref{lm:evCD}.
We have thus proven the necessity part.

Now we proceed with sufficiency. Here it suffices to prove that given a \twr $\TWR$ of type $\mu$ compatible with a level graph $\lG$ such that the evaluation morphism vanishes identically, we can construct a one-parameter family of smooth curves in $\DR$ degenerating to $\TWR$. By \Cref{prop:twistEx} there exists a twist $\TWDR$ of $\TWR$ which is of type $\eptn$ for some partition extending $\ptn$ compatible with some level graph $\lG'$. Again, by \Cref{lm:evCD} we have $\ev_{\psf}^{(i)}=0$ for all $i$ and thus $\TWDR\in\overline{\DRDc}\subseteq \PP\Xi\Mgnbar[g,n'](\eptn)$ by \Cref{prop:BDDRc}.
Therefore, there exists a holomorphic map $\phi':\Delta\to \overline{\DRDc}$ which is generically contained in $\DRDc$ and $\phi'(0)=\TWDR$. Composing with $\pi:\PP\Xi\Mgnbar[g,n'](\eptn)\to\Mgnbar[g,n']$ we obtain a map $\phi:=\pi\circ\phi':\Delta\to\overline{\DRD}$ which is generically contained in $\DRD$ and $\phi(0)=(\stC,\umS)$.

\end{proof}

\section{Admissible covers versus twistable rational functions}\label{sec:Adm}
In this section we discuss the relationship between using admissible covers and twistable rational functions to describe the closure of double ramification loci.
%The approach using admissible covers is probably known to the expert, but we couldn't find it in the literature.
 As a consequence we see that the existence of an admissible cover can be encoded in the level graph of a \twr, a result that to the best of our knowledge is new.

We start by setting up our notation for admissible covers. As before let $\ptn$ be a partition of zero.
The {\em Hurwitz scheme} $\Hur$ parameterizes maps $f:X\to\PP^1$ with ramification over $0$ and $\infty$ prescribed by $\mu$, and simple ramification otherwise. We mark on $X$ all preimages of branch points of $f$, see for example \cite{JKK},  where the same variant of admissible covers is used.

We now recall admissable covers, see for example the original reference \cite{HM}. Note that here we do not require the ramification to be simple.
\begin{definition}[Admissible cover]
Let $(D,\umS)$ be a stable $n$-pointed nodal curve of genus $0$. An {\em admissible cover} $f:C\to D$ of degree $d$ is a finite degree $d$ morphism such that
\begin{itemize}
\item $C$ is nodal and every node of $C$ maps to a node of $D$,
\item the restriction of $f$ to $D_{gen}$ is \'{e}tale, where $D_{gen}$ is the complement of the marked zeros and nodes,
\item if a node of $C$ identifies $x$ and $y$, then
\[
\mult_x f=\mult_y f,
\]
where $\mult_x f-1$ is the ramification index of $f$ at $x$.
\end{itemize}
We additionally require that the marked points in $C$ are exactly  the preimages of marked points in $D$. The {\em type} of $f$ records ramification profile of $f$ at all marked points which are not nodes.
\end{definition}

We let $\Adm$ be the stack of all admissible covers of type $(\mu,1,\ldots,1)$ where we additionally mark all preimages of marked points in $D$. Then $\Adm$ is a proper Deligne-Mumford stack and $\Hur$ is a dense open substack.
Note that $\Adm$ is in general not smooth, but its normalization is, see for example \cite{ AVC}, where the authors also give a modular interpretation of the normalization in terms of twisted $S_d$-covers. Our notation differs from  (loc. cit.) in that we  mark all preimages of the marked points of $D$.

Since $\Hur$ is an open dense substack of $\Adm$, we have the following.

\begin{proposition}
A stable $n$-pointed curve $(X,\umS)$ is contained in the closure of $\DR$ if and only if there exists a marked stable curve ($X',\umS')$ such that
\begin{itemize}
\item there exists an admissible cover  $X'\to D$ of type $(\mu,1,\ldots,1)$,
\item the stabilization of the prestable curve $(X',f^{-1}(\{0,\infty\}))$ is isomorphic to $(X,\umS)$.
\end{itemize}
\end{proposition}

\begin{proof}
Generically a rational map in a double ramification locus has simple ramification and thus we can lift a family in $\DR$ to $\Hur(\sigma)$ where $\sigma=(\mu,1,\ldots,1)$. Since $\Hur$ is dense in $\Adm$, we get the desired result.
\end{proof}

We have thus described the closure in two seemingly different ways, with admissible covers and with twistable rational functions.
As a result we obtain a correspondence between the existence of  a suitable admissible cover on a prestable model of $X$  and the existence of  a certain \twr on $X$.
In general this correspondence is far from being one-to-one.

\begin{example}[Twistable rational functions and admissible covers do not determine each other]\label{ex:Cherry}
Even  on a smooth curve $X$ there can be multiple admissible covers. For example we can attach unstable components to $X$ such that the resulting prestable curve $X'$ curve maps to the union of two rational curves attached at a single node and $X$ is the stabilization of $X'$.

On the other hand, given a fixed admissible cover there can be multiple level graphs compatible with it. For example we can construct an admissible cover on a ``cherry level graph'', see \Cref{fig:Cherry}. The two level graphs $\lG_1$ and $\lG_2$ are obtained from each other by tilting the lower levels components and are supported on the same dual graph $\Gamma$.
Suppose now we have a \twr $\TWR$ such that all marked zeros are on the top level component. In this case the vanishing of the evaluation morphism of $\lG_1$ for a \twr on $X$ is equivalent to the vanishing on $\lG_2$.
So given any admissable cover on $X$ we can then construct two different \twrs, one on $\lG_1$ and one on $\lG_2$.
 By choosing different genera for the irreducible components on the lower levels one can make sure that $\lG_1$ and $\lG_2$ are not isomorphic level graphs.
\begin{figure}
\centering
\includegraphics[scale=1.2]{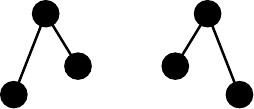}
\caption{The cherry graph from \Cref{ex:Cherry} with two level graphs $\lG_1$ and $\lG_2$}
\label{fig:Cherry}
\end{figure}
\end{example}

It appears very difficult to try to approach the correspondence between admissable covers and \twrs directly, since the combinatorial difficulties are very significant. Given a \twr one needs to insert unstable components in order to obtain a finite morphism  but there are many different ways one can add unstable components. And in the other direction, given an admissable cover one needs to find a level graph on a given dual graph such that the evaluation morphism vanishes. The number of possible level graphs on a given dual grows exponentially, see for example \cite{diffstrata} for the problem of enumerating all possible level graphs.

The correspondence can be generalized to admissible covers with arbitrary fixed ramification multiplicities. The proof remains the same, at least as long all poles of $\stf$ are marked. In the case where the poles of $\stf$ are marked it doesn't seem possible to state the theorem purely on the stable curve $X$ but one can only phrase it in terms of a prestable model. In general the bookkeeping is tedious since one needs to remember which branch points lie in the same fiber and depending on that the evaluation morphisms changes. We leave the details to the reader.

\section{Examples}\label{sec:ex}
We now illustrate the main theorem in a few examples and showcase some of the features of \twrs.
We will use the following conventions to depict level graphs. The legs denote the marked zeros and poles and critical points of the rational functions. Dashed legs correspond to marked critical points which are not marked zeroes or poles in $\umS$.
The label associated to a half-leg is the order of vanishing of the exact differential $df$ (and {\em not} the order of the rational function $f$) at the corresponding marked point. The label of a vertex is the genus of the normalization of the corresponding irreducible component. If a vertex is unlabeled, it corresponds to an irreducible component of genus zero.

 The decorations on the edges are  the numbers $\ord_{\mN^+} df$ and $\ord_{\mN^-} df$ respectively. The level order is implicit in the figures: the highest vertices correspond to the top level, lower levels are drawn below and vertices of  the same height are of the same level.
From now on $\TWR$ always denotes a collection of rational functions on a stable curve such that $(\stC,\umS)$ is contained in the closure of some \dr $\DR$ and $\TWR$ is a \twr for some partition $\eptn$ extending $\mu$ compatible with some level graph $\lG$.

\begin{example}[Unmarked zeros]\label{ex:unm}
As a first example we consider a degeneration of $\DRg[2](1^3,-3)$, depicted in \Cref{fig:dollar}.
\begin{figure}[H]
\begin{subfigure}[c]{0.4\textwidth}
\centering
\includegraphics[scale=1]{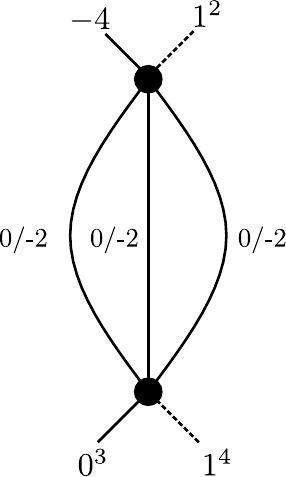}
\caption{Level graph of a dollar curve}
\label{fig:dollar}
\end{subfigure}
~
\begin{subfigure}[c]{0.6\textwidth}
\centering
\includegraphics[scale=0.5]{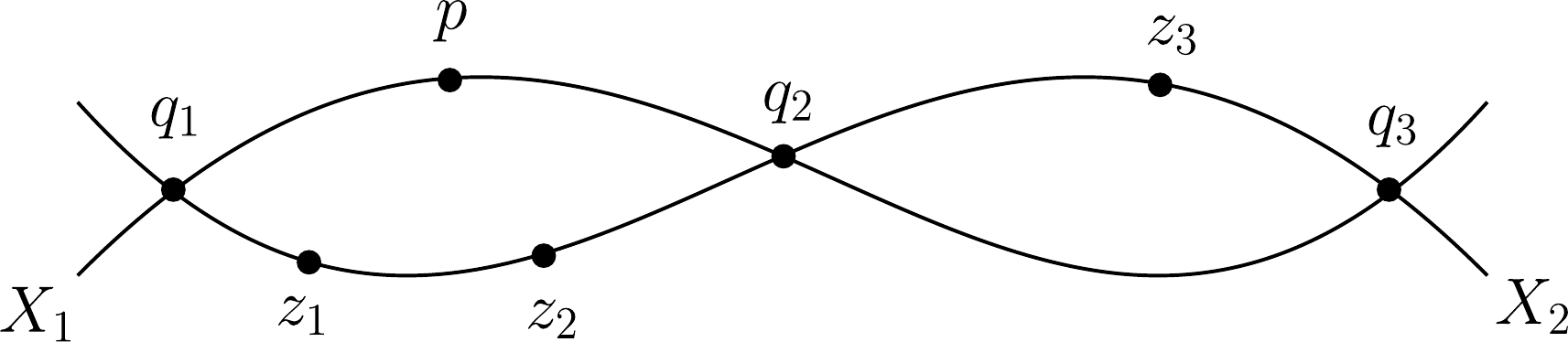}
\caption{A dollar curve}
\end{subfigure}
\caption{The dollar curve from \Cref{ex:unm}}
\end{figure}

 It is an example of a ``dollar curve", which we discuss more thoroughly in \Cref{ex:Dollar}.
We note that the top level component does not contain any marked zeros. In particular we see that both types of components mentioned in \Cref{rem:unmarked} do appear.

The evaluation morphism can be computed as follows. Let $\gamma_1$ be the path $(X_1,q_2,X_2,q_1,X_1)$ passing once through $q_1$ and $q_2$ and similarly let $\gamma_2$ be the path $(X_1,q_3,X_2,q_2,X_1)$ passing through $q_2$ and $q_3$ once.
Then
\[
\begin{split}
\im \ev_{f}^{(0)}=\langle \ev_{f}^{(0)}(\gamma_1), \ev_{f}^{(0)}(\gamma_2)\rangle,\\
\ev_{f}^{(0)}(\gamma_1)=f_1(q_2^+)-f_1(q_1^+),\\
\ev_{f}^{(0)}(\gamma_2)=f_1(q_3^+)-f_1(q_2^+)
\end{split}
\]
and the vanishing of the evaluation morphism is equivalent to $f_1(q_1^+)=f_1(q_2^+)=f_1(q_3^+)$, i.e. $\{q_1^+, q_2^+, q_3^+\}$ lie in the same fiber of $f_1$.
Note that the evaluation morphism automatically vanishes at level $-1$ in this example since there are no non-trivial relative paths contained in level $-1$.
Note that such a rational function $f_1:X\to\PP^1$   always exists. Similarly on $X_2$ the rational function $f_2$ has three simple zeroes, three simple poles and four remaining simple branch points. Such a rational functions always exists as well.

\end{example}

\begin{example}[Level graphs with horizontal nodes]\label{ex:hor}

We have seen that level graphs for \twdrs have no horizontal nodes. The next example shows that the level graph of a \twr can have horizontal nodes.
We consider again a degeneration of $\DRg[2]{(1^3,-3)}$, see \Cref{fig:ex1}.

\begin{figure}[H]
\begin{subfigure}{0.9\textwidth}
\centering
\includegraphics[scale=0.8]{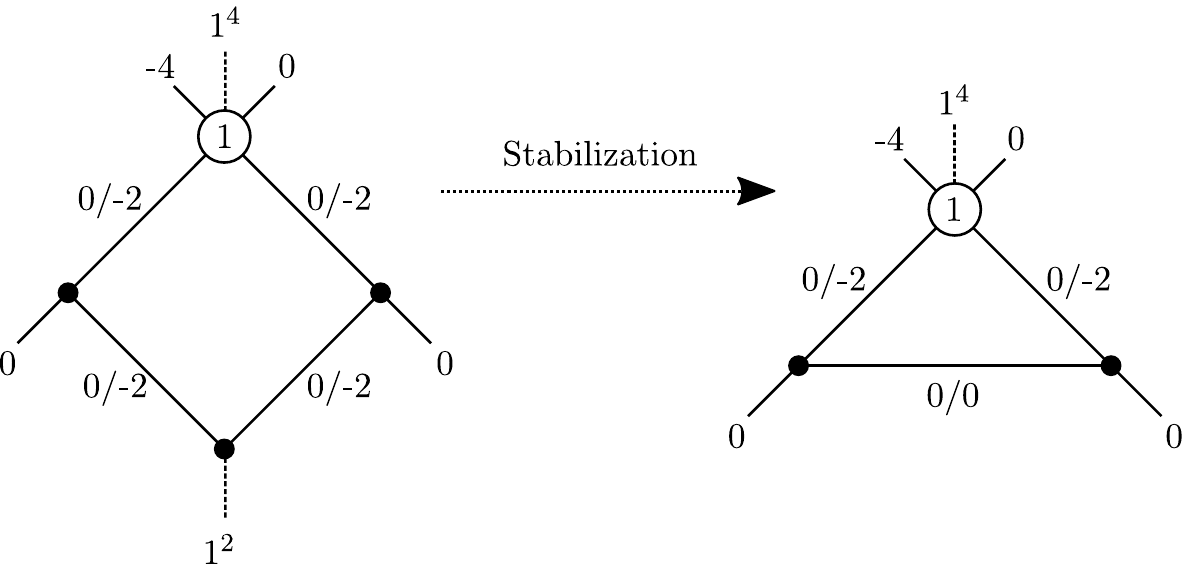}
\caption{The level graph $\lG'$ and its stabilization $\lG$}
\label{fig:ex1a}
\end{subfigure}

\begin{subfigure}[t]{.8\textwidth}
\centering
\includegraphics[scale=0.15]{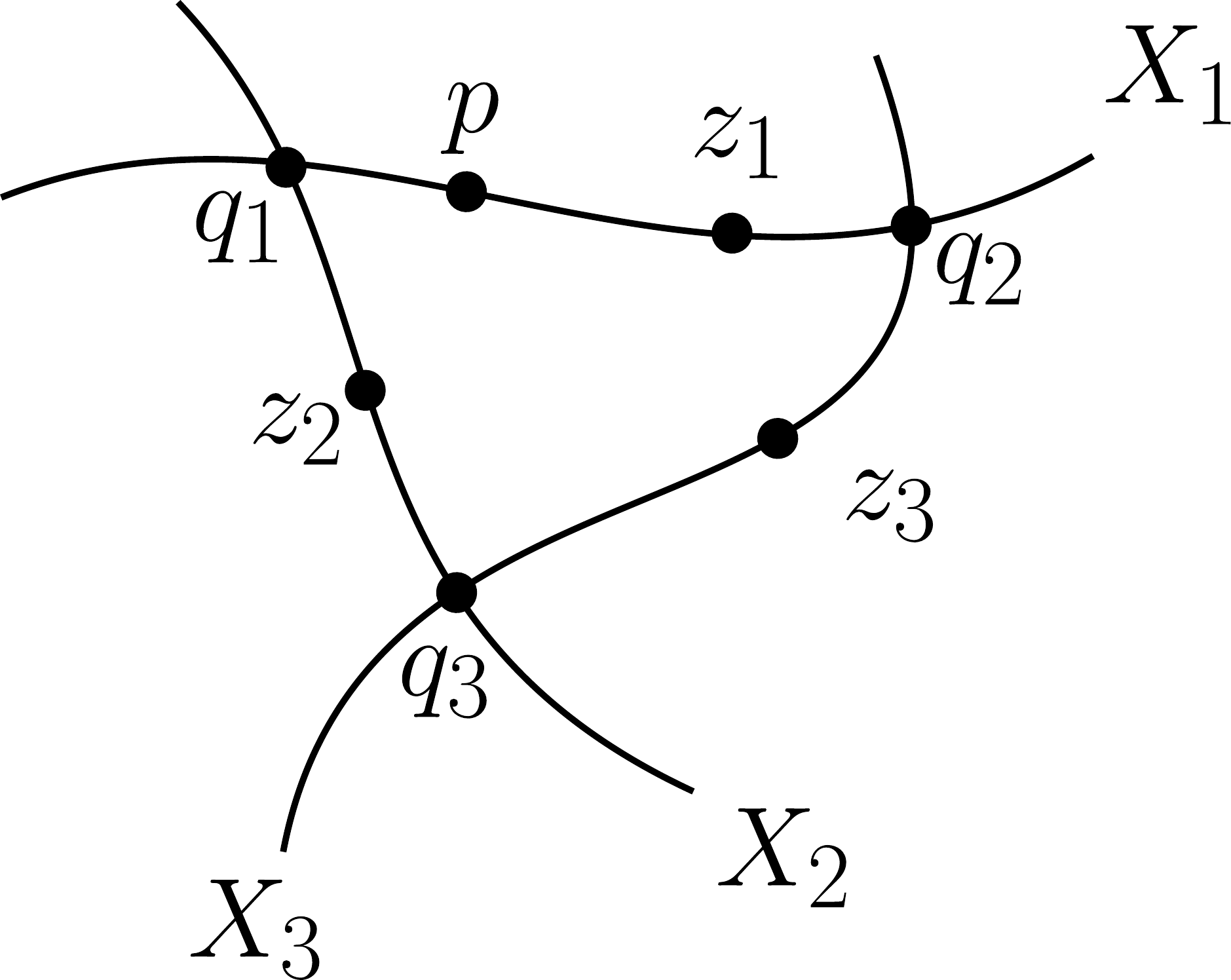}
\caption{A stable curve $X$ with level graph $\lG$}
\label{fig:ex1b}
\end{subfigure}
\caption{The dual graphs from \Cref{ex:hor}}
\label{fig:ex1}
\end{figure}
The level graph $\lG'$ is the level graph of a \twdr and $\lG$ the level graph of its stabilization. Note that $\lG'$ has three levels, under stabilization the bottom vertex is contracted and thus $\lG$ has only two levels with a horizontal edge.
Let $\TWR$ be the twist of  $\TWDR$ which is contained in the closure of $\DRg[2]{(1^3,-3)}$.
On the genus $1$ component $X_1$ the rational function $f_1:X_1\to \PP^1$ is a degree three map which has a pole of order three at $p$ and simple zero at $z_1$.
 For a fixed pointed elliptic curve $(X_1,z_1,p,q_1^+,q_2^+)$ there is a $h^0(\calO_{X_1}(3p-z_1))=2$-dimensional family of such maps.

We now compute the top level evaluations $\ev^{(0)}_{\stf},\,\ev^{(-1)}_{\stf}$.
Let $\gamma_{ij}\in H_1(\lG,\umZ)$ be the path $(z_i,v_i,v_j,z_j)$ connecting the marked zeros $z_i$ and $z_j$ by only passing through a single node. Then the image of the evaluation map is generated by the paths $\gamma_{ij}$, in particular
\[
\begin{split}
\im \ev^{(0)}_{\stf}&=\langle  \ev^{(0)}_{\stf}(\gamma_{12}),\ev^{(0)}_{\stf}(\gamma_{13})\rangle_{\ZZ},\\
\im \ev^{(-1)}_{\stf}&=\langle  \ev^{(-1)}_{\stf}(\gamma_{23})\rangle_{\ZZ}.
\end{split}
\]
We then compute
\[
\begin{split}
\ev^{(0)}_{\stf}(\gamma_{12})&= f_1(z_1)-f_1(q_1^+)=-f_1(q_1^+),\\
\ev^{(0)}_{\stf}(\gamma_{13})&= f_1(z_1)-f_1(q_2^+)=-f_1(q_2^+),\\
\ev^{(-1)}_{\stf}(\gamma_{23})&= f_2(q_3^+)-f_3(q_3^-).
\end{split}
\]
The vanishing of $\ev^{(0)}_{\stf}(\gamma_{12})$ and $\ev^{(0)}_{\stf}(\gamma_{13})$ then implies that
\[
(f_1)=z_1+q_1^++q_2^+-3p
\]
and thus $(X_1,(z_1,q_1^+,q_2^+,p))\in \DRg[1]{(1^3,-3)}$.

Finally, the vanishing of $\ev^{(1)}_{\stf}(\gamma_{23})$ implies
\[
f_2(q_3^+)-f_3(q_3^-)=0.
\]
Both $f_2$ and $f_3$ are isomorphisms to $\PP^1$ sending $z_2$ and $z_3$ to zero, respectively. In particular we have $f_2(q_3^+)\neq 0,f_3(q_3^-)\neq 0$ and by rescaling $f_3$  we can always achieve that $f_2(q_3^+)-f_3(q_3^-)=0$.
Thus in this case the only condition for $(X,\umS,\stf)$ to be in the closure of $\DRg[2](1^3,-3)$ is given by $(X_1,(z_1,q_1^+,q_2^+,p))\in \DRg[1]{(1^3,-3)}$.

\end{example}

For the following examples we will omit the remaining critical points from the level graphs and only draw the marked zeros and poles of $\stf$.
\begin{example}[The partial ordering is not determined by the vanishing orders]\label{ex:paOrder}
\begin{figure}[H]
\centering
\includegraphics[scale=1.1]{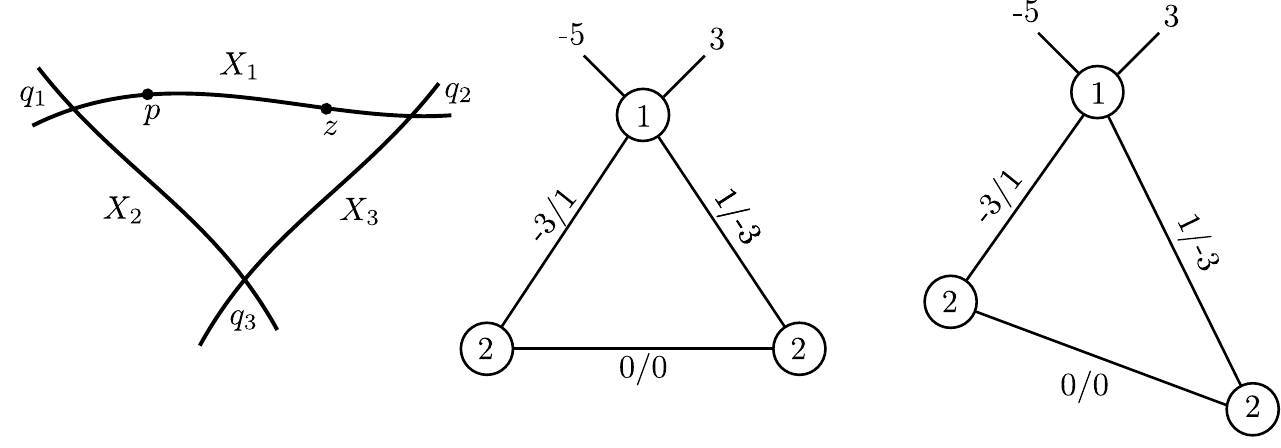}
\caption{The different partial orders in \Cref{ex:paOrder}}
\label{fig:paOrder}
\end{figure}

In the case of twisted differentials $(X,\eta)$, a twisted differential induces a unique partial ordering on the dual graph.
We will now see that the same is not true for \twrs. More precisely, we will construct a \twr $\TWR$ which lies in the boundary of  $\DRg[5](4,-4)$ such that there are two different level graphs, inducing a different partial order, such that $\TWR$ is compatible with both level graphs.
By Condition $(3)$ in \Cref{def:twr} this is only possible if there is a node such that the rational function is holomorphic at both preimages of the node.

We let $\lG_1$ be the dual graph in the middle of \Cref{fig:paOrder} and $\lG_2$ the one on the right.
 Note that the two partial orderings induced by $\lG_1$ and $\lG_2$ are different since $\lG_1$ has a horizontal node and $\lG_2$ has not.
To ensure that we can find a \twr compatible with both level graphs we now unravel the vanishing of the evaluation morphism for $\lG_1$ and $\lG_2$.
Let $\gamma_1$ be the path $(q_1,q_2,q_3)$. Then
\[
\begin{split}
\im \ev^{(0)}(f)(\lG_k)=\langle \ev^{(0)}_f(\lG_k)(\gamma_1)\rangle \text{ for } k=1,2\
\end{split}
\]
and the evaluation morphism vanishes automatically  at all other levels, since in both cases the level graph has no loops in lower levels and also no marked zeros on components of lower levels.
We then compute
\[
\begin{split}
\ev^{(0)}_f(\lG_1)(\gamma_1)&=\ev^{(0)}_f(\lG_2)(\gamma_1)=f_1(q_2^+)-f_1(q_1^+).
\end{split}
\]
Thus for both $\lG_1$ and $\lG_2$ on $X_1$ the rational function $f_1$ is of degree $4$, $q_1^+, q_2^+$ are in the same fiber of $f_1$, and $f_1$ is simply ramified at $q_1^+$ and $q_2^+$.
Note that such a rational function always exists, for example by \cite[Thm. 2.7]{PPHurwitz}.
Since the vanishing of the evaluation morphism is controlled by the top level, every \twr that is compatible with $\lG_1$ is automatically compatible with $\lG_2$.

%The two different level graphs give rise to different admissible covers
%The example can be modified slightly to produce an example where the same admissible cover is compatible with  two different level graphs.
%First we one should make the example asymmetric by changing the genus of one of the two genus $2$ components

%
%On both genus $2$ components $X_2$ and $X_3$ is the quotient by the hyperelliptic involution.
%For $\lG_1$ the vanishing of the evaluation morphism implies
%\[
%f_2(q_3^+)-f_2(q_4^+)=f_3(q_3^-)-f_3(q_4^-).
%\]
%Either $q_3^+$ and $q_4^+$ are conjugate under the hyperelliptic involution, in which case the same is true for $q_3^-$ and $q_4^-$, or otherwise we can rescale $f_3$ so that the condition is vacuous.
%For the remaining level graph $\lG_2$ the vanishing of the evaluation morphism only means that $q_3^+$ and $q_4^+$ are conjugate under the hyperelliptic involution.

\end{example}

\begin{example}[The evaluation morphism depends on the level graph]\label{ex:lvldep}

The notion of \twr does not depend on the full structure of a level graph. The level graph only comes into play when considering the evaluation morphism. This example
%, which is very similar to \Cref{ex:paOrder},
 shows that there can be two different level graphs inducing the same partial ordering but the condition imposed by the evaluation morphism is different.
In \Cref{fig:lvldep} we see a degeneration of $\DRg[8](4,-4)$.

\begin{figure}[H]
\centering
\includegraphics[scale=1.0]{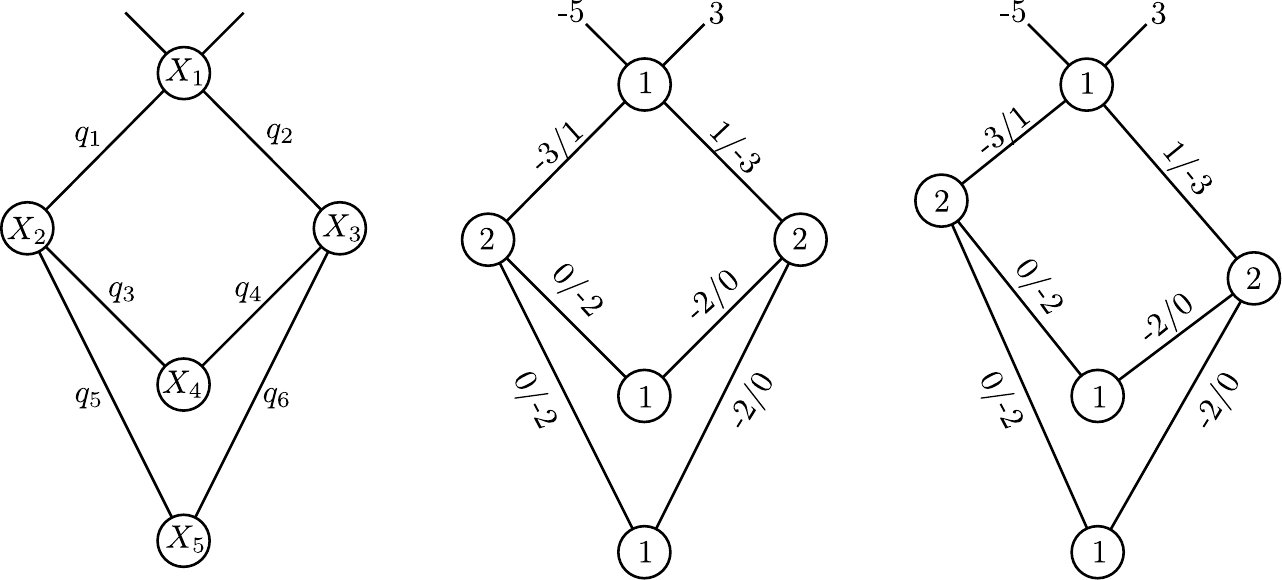}
\caption{The different level graphs for \Cref{ex:lvldep}}
\label{fig:lvldep}
\end{figure}
%We let $\gamma_1$ be the path $(q_1,q_2,q_3,q_4)$ and $\gamma_2$ be the path $(q_3,q_4,q_5,q_6)$. The evaluation morphism is only non-trivial in levels $0$ and $-1$. For both $k=1,2$ the image of $\ev_{f}^{(0)}(\lG_k)$ and  $\ev_{f}^{(-1)}(\lG_k)$ is generated by $\gamma_1$ and $\gamma_2$, respectively.
%Then, exactly as in  \Cref{ex:paOrder}, in the top level the vanishing of the evaluation morphism is independent of $k$ and is equivalent to $f_1(q_1^+)=f_1(q_2^+)$.
% we have \[
%\begin{split}
%\im \ev^{(0)}(f)(\lG_k)&=\langle \ev^{(0)}_f(\lG_k)(\gamma_1)\rangle,\\
%\im \ev^{(-1)}(f)(\lG_k)&=\langle \ev^{(-1)}_f(\lG_k)(\gamma_2)\rangle \text{ for } k=1,2.\\
%\end{split}
%\]
%and
On the top level the vanishing of the evaluation morphism is independent of $k$ and equivalent to $f_1(q_2^+)=f_1(q_1^+)$. Exactly as in \Cref{ex:paOrder} we see that there exists  such a rational function.
For  level $-1$ we have
\[
\begin{split}
\ev^{(-1)}_f(\lG_1)(\gamma_2)&= f_2(q_3^+)-f_2(q_5^+)+f_3(q_6^+)-f_3(q_4^+),\\
\ev^{(-1)}_f(\lG_2)(\gamma_2)&= f_2(q_3^+)-f_2(q_5^+).
\end{split}
\]
Thus in the case of $\lG_2$, for the evaluation morphism to vanish the preimages of the nodes $q_3^+$ and $q_5^+$ have to be conjugate under the hyperelliptic involution, while for $\lG_1$ the condition $f_2(q_3^+)-f_2(q_5^+)=f_3(q_4^+)-f_3(q_6^+)$ can be satisfied without $q_3^+$ and $q_5^+$ being conjugate under the hyperelliptic involution.

In light of our discussion about admissible function the condition $f_2(q_3^+)-f_2(q_5^+)+f_3(q_6^+)-f_3(q_4^+)$ is, at first glance, surprising since it does not simply state that two points lie in the same fiber. But via an automorphism $z\mapsto z+c$ we can arrange that $f_2(q_3^+)=f_3(q_4^+)$. In which case $f_2(q_3^+)-f_2(q_5^+)+f_3(q_6^+)-f_3(q_4^+)$ simplifies to $f_2(q_5^+)=f_3(q_6^+)$.
Now those two conditions can be used to construct an admissible cover by mapping components of the same level to the same copy of $\PP^1$. Note that one has to attach additional unstable components for this.

\end{example}

Our next example is a thorough discussion of so-called {\em dollar curves}, i.e. two irreducible curves meeting transversally at three nodes. We study all possible level graphs and describe explicitly the conditions to be in the closure of \dr.  For curves of compact type, or for curves with at most one-separating node, one can extend the Abel-Jacobi map to describe the closure of \dr, see \cite{GZ}. Thus dollar curves are a natural next step to illustrate our techniques.

\begin{example}[``Dollar curves"]\label{ex:Dollar}
We let $X=X_1\cup X_2$ be a stable curve in the boundary of $\DR$ with one component $X_1$ of genus $g_1$ and the other component $X_2$ of genus $g_2:=g-g_1-2$ meeting at three nodes $q_1, q_2, q_3$, see \Cref{fig:Dollard}.
\begin{figure}
\begin{subfigure}{.3\textwidth}
  \centering
  % include first image
  \includegraphics[scale=1.5]{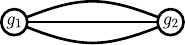}
 % \caption{$\lG_1$}
  \label{fig:Dollara}
\end{subfigure}
\begin{subfigure}{.15\textwidth}
  \centering
  % include first image
  \includegraphics[scale=1.5]{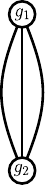}
  %\caption{$\lG_2$}
  \label{fig:Dollarb}
\end{subfigure}
\begin{subfigure}{.2\textwidth}
  \centering
  % include first image
  \includegraphics[scale=1.5]{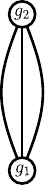}
 % \caption{$\lG_3$}
  \label{fig:Dollarc}
\end{subfigure}~
\begin{subfigure}{.2\textwidth}
  \centering
  % include first image
  \includegraphics[scale=0.2]{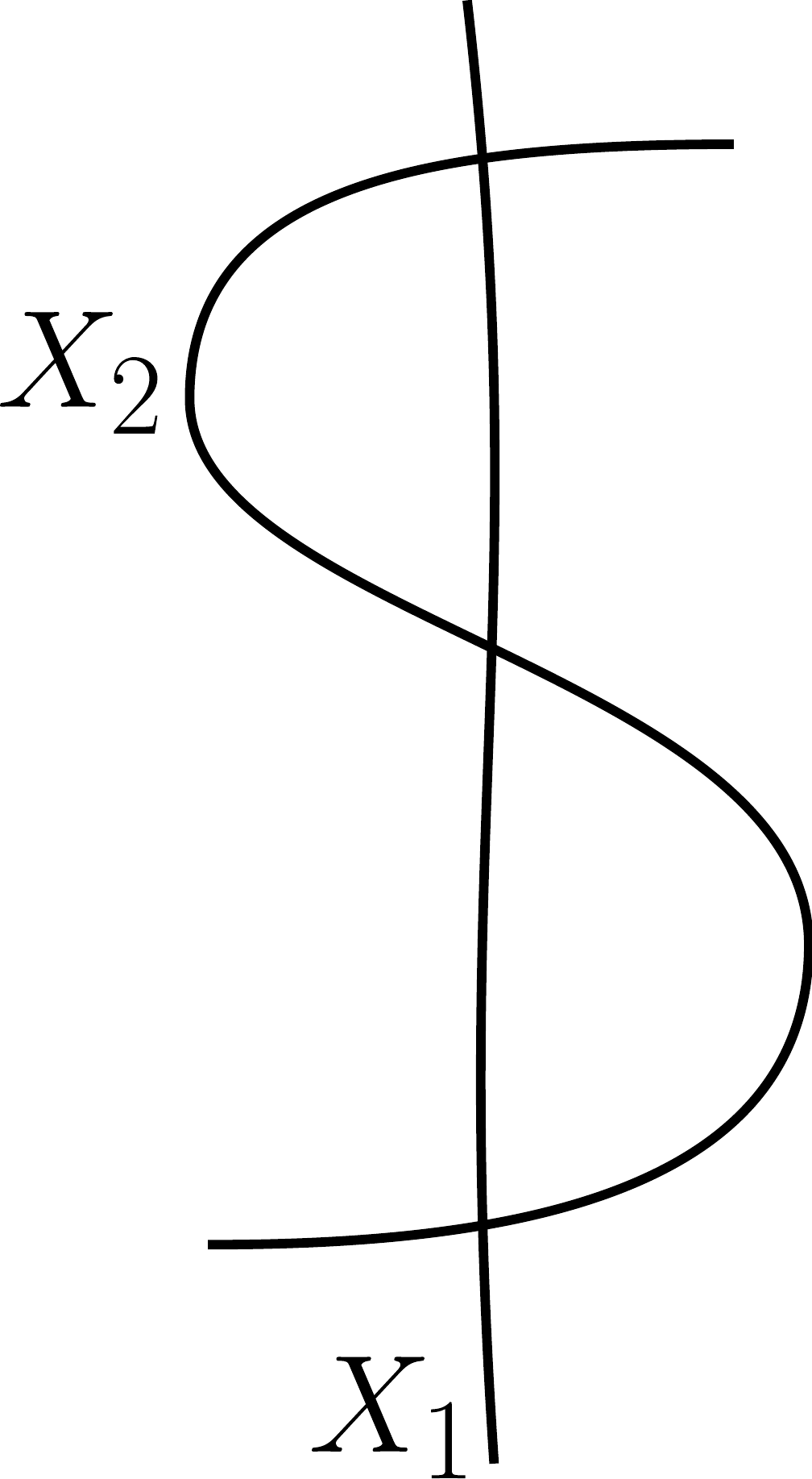}
 %% \caption{The stable curve $X$}
  \label{fig:Dollard}
\end{subfigure}
\caption{The possible level graphs $\lG_1,\lG_2, \lG_3$  from \Cref{ex:Dollar}} and the resulting stable curve $X$
\label{fig:Dollar}
\end{figure}
Without labeling the legs, there are only three possible level graphs, see \Cref{fig:Dollar}.
The discussion for $\lG_2$ and $\lG_3$ is completely analogous and we thus only discuss $\lG_1$ and $\lG_2$.
We let $\umS_1$ and $\umS_2$ be all marked zeros or poles of $f$ contained in $X_1$ and $X_2$, respectively. Similarly, we let $\umN^+=(q_1^+,q_2^+,q_3^+)$ and $\umN^-=(q_1^-,q_2^-,q_3^-)$ be the preimages of nodes contained in $X_1$ and $X_2$ respectively.
The following condition are necessary for a \twr.

\begin{enumerate}
\item If $\ord_{q_l^\pm}\stf<0$ then $v(q_l^{\pm})$ is of level $-1$.
\item $\ord_{q_l^+} (df_1)+\ord_{q_l^-} (df_2)\geq -2$ for $l=1,2,3$.
\end{enumerate}
We now discuss the vanishing of the evaluation morphism for the different level graphs.

{\em The dual graph $\lG_1$:} In this case both components $X_1$ and $X_2$ are of top level and thus $\ord_{q_l^{\pm}}\stf\geq 0$ for $l=1,2,3$.
This is only possible if $\ord_{q_l^{\pm}} \stf=0$ for $l=1,2,3$.
Thus $f_1$ and $f_2$ have zeros and poles  only at the marked points.
In particular
%$(X_k,\umS_k)\in \DRg[g_k](\ptn_1)$ and furthermore,
both components $X_1$ and $X_2$ contain marked zeros of $\stf$.

It remains to compute the evaluation morphism.
Let $\gamma_l$ be a path connecting a marked zero of $X_1$ and a marked zero of $X_2$ through the node $q_l$. Then the first homology $H_1(\lG_1,\umZ)$ of the dual graph relative to the marked zeros of $\stf$  is generated by $\gamma_1,\gamma_2$ and $\gamma_3$.
We then have
\begin{equation*}
\begin{split}
\ev_{\stf}^{(0)}(\gamma_l)&=f_1(q_l^+)-f_2(q_l^-) \text{ for } l=1,2,3\,,
\end{split}
\end{equation*}
and thus the vanishing of the evaluation morphism is equivalent to
\[
f_1(q_l^+)=f_2(q_l^-) \text{ for } l=1,2,3.
\]

Since no preimage of a node $q_l^{\pm}$ is a zero or pole of $\stf$, we can always rescale $f_2$ such that the condition $f_1(q_1^+)=f_2(q_1^-)$ is satisfied. But then the conditions for $l=2,3$ impose non-trivial restrictions.

{\em The dual graph $\lG_2$:} We now consider the second case, where $X_1$ is of top level and $X_2$ of level $-1$.
The main difference is that now potentially one component contains no marked zeros of $\stf$. We thus distinguish two cases.

{\em Case 1: Both components contain a marked zero}\\
In this case $H_1(\lG,\umS)$ is again generated by $\gamma_l,\, l=1,2,3$.
Since $X_2$ is of lower level the evaluation morphism now computes as
\[
\ev_{\stf}^{(0)}(\gamma_l)=f_1(q_l^+) \text{ for } l=1,2,3
\]
and thus the vanishing of  the evaluation morphism is equivalent to $\stf_1(q_l^+)=0$ for $l=1,2,3$. In particular all preimages $q_l^+$ are zeros for $\stf_1$.
On $X_2$ there are no further restrictions.

{\em Case 2: There exists a component without  marked zeros}\\
The evaluation morphism now computes differently, since there is no relative path crossing exactly one node.
We let $\gamma_{ab}$ to be the path $(X_1,q_a,X_2,q_b,X_1)$ be a path crossing only $q_a$ and $q_b$ once. Then $\gamma_{12},\gamma_{13}$ generate $H_1(\lG,\umZ)$ and
\[
\begin{split}
\ev_{\stf}^{(0)}(\gamma_{12})&=f_1(q_1^+)-f_1(q_2^+),\\
\ev_{\stf}^{(0)}(\gamma_{13})&=f_1(q_1^+)-f_1(q_3^+).
\end{split}
\]
Thus the vanishing of the evaluation morphism is equivalent to
\[
f_1(q_1^+)=f_1(q_2^+)=f_1(q_3^+).
\]
\end{example}
Combining our analysis we have shown that

\begin{proposition}
A marked curve $(X,\umS)$ with dual graph $\lG$ lies in the closure of $\DR$ if and only if there exist rational functions $f_k$ on $X_k$ for $k=1,2$ satisfying the following conditions.
\begin{enumerate}
\item $\ord_{\mS}= \ptn_k$  for all $k$.
\item $\ord_{q_l^+} (df_1)+\ord_{q_l^-} (df_2)\geq -2$ for $l=1,2,3$.
 \item Either \begin{gather*}
 f_1(q_1^+)=f_1(q_2^+)=f_1(q_3^+),\, \ord_{q_l^{+}} f_1\geq 0 \textrm{ for } l=1,2,3 \text{ or }\\
 f_2(q_1^-)=f_2(q_2^-)=f_2(q_3^-), \ord_{q_l^{-}} f_2\geq 0 \textrm{ for } l=1,2,3 \text{ or }\\
f_1(q_l^+)=f_2(q_l^-),\, \ord_{q_l^{\pm}} f=0 \textrm{ for } l=1,2,3.
\end{gather*}
\end{enumerate}
\end{proposition}
We see that in this example we get a short list of conditions that characterize whether a curve lies in the closure of a \dr. On each irreducible component one has a rational function  contained in a double ramification locus, where potentially we have some additionally marked points that are not zeros or poles of $\stf$, together with some compatibility conditions between $f_1$ and $f_2$.    All the conditions either restrict the ramification profile on one  irreducible components or state that some marked points on different components are mapped to the same point on $\PP^1$.

\end{document}